\documentclass[11pt,reqno]{amsart}
\usepackage[T1]{fontenc}
\usepackage[utf8]{inputenc}
\usepackage{newtxtext,newtxmath}
\usepackage{microtype,mathtools,mathrsfs,enumitem,booktabs,longtable,array,amscd}
\DeclareFontFamily{U}{rsfs}{\skewchar\font127}
\DeclareFontShape{U}{rsfs}{m}{n}{<-6.5>rsfs5 <6.5-8>rsfs7 <8->rsfs10}{}
\newcolumntype{P}[1]{>{\raggedright\arraybackslash}p{#1}}
\usepackage{aliascnt,needspace}
\usepackage{xcolor}
\usepackage[colorlinks=true,linkcolor=blue!45!black,citecolor=blue!45!black,
 urlcolor=blue!45!black,unicode=true]{hyperref}
\usepackage[nameinlink,noabbrev,capitalise]{cleveref}
\numberwithin{equation}{section}
\newtheorem{theorem}{Theorem}[section]
\newaliascnt{proposition}{theorem}\newtheorem{proposition}[proposition]{Proposition}\aliascntresetthe{proposition}
\newaliascnt{lemma}{theorem}\newtheorem{lemma}[lemma]{Lemma}\aliascntresetthe{lemma}
\newaliascnt{corollary}{theorem}\newtheorem{corollary}[corollary]{Corollary}\aliascntresetthe{corollary}
\theoremstyle{definition}
\newaliascnt{definition}{theorem}\newtheorem{definition}[definition]{Definition}\aliascntresetthe{definition}
\theoremstyle{remark}
\newaliascnt{remark}{theorem}\newtheorem{remark}[remark]{Remark}\aliascntresetthe{remark}
\newaliascnt{example}{theorem}\newtheorem{example}[example]{Example}\aliascntresetthe{example}
\crefname{proposition}{proposition}{propositions}\Crefname{proposition}{Proposition}{Propositions}
\crefname{lemma}{lemma}{lemmas}\Crefname{lemma}{Lemma}{Lemmas}
\crefname{corollary}{corollary}{corollaries}\Crefname{corollary}{Corollary}{Corollaries}
\crefname{definition}{definition}{definitions}\Crefname{definition}{Definition}{Definitions}
\crefname{remark}{remark}{remarks}\Crefname{remark}{Remark}{Remarks}
\crefname{example}{example}{examples}\Crefname{example}{Example}{Examples}
\setlist[enumerate]{label=\textup{(\roman*)},leftmargin=2.2em,itemsep=3pt}
\allowdisplaybreaks[2]
\newcommand{\Q}{\mathbf Q}\newcommand{\Z}{\mathbf Z}\newcommand{\R}{\mathbf R}\newcommand{\C}{\mathbf C}
\newcommand{\Qbar}{\overline{\mathbf Q}}\newcommand{\Gm}{\mathbf G_m}\newcommand{\Ga}{\mathbf G_a}
\newcommand{\A}{\mathcal A}\newcommand{\G}{\mathcal G}\newcommand{\HH}{\mathcal H}
\newcommand{\PP}{\mathcal P}\newcommand{\OO}{\mathcal O}\newcommand{\TT}{\mathbb T}
\newcommand{\MM}{\mathscr M}\newcommand{\BB}{\mathscr B}\newcommand{\DD}{\mathscr D}
\newcommand{\Gr}{\operatorname{gr}}\newcommand{\MT}{\operatorname{MT}}
\newcommand{\Lie}{\operatorname{Lie}}\newcommand{\im}{\operatorname{im}}\newcommand{\Hom}{\operatorname{Hom}}
\newcommand{\End}{\operatorname{End}}\newcommand{\Pic}{\operatorname{Pic}}
\newcommand{\rank}{\operatorname{rank}}
\newcommand{\Supp}{\operatorname{Supp}}\newcommand{\divi}{\operatorname{div}}
\newcommand{\trdeg}{\operatorname{trdeg}}\newcommand{\Ru}{R_{\!u}}
\newcommand{\tors}{\mathrm{tors}}
\newcommand{\Zar}{\mathrm{Zar}}
\newcommand{\ol}[1]{\overline{#1}}\newcommand{\ddc}{dd^c}
\newcommand{\eps}{\varepsilon}\newcommand{\bL}{\overline L}\newcommand{\bQ}{\overline Q}
\newcommand{\bN}{\overline N}\newcommand{\bT}{\overline T}\newcommand{\bD}{\overline D}
\newcommand{\bH}{\overline H}\newcommand{\bO}{\overline{\OO}}
\newcommand{\bTheta}{\overline\Theta}
\newcommand{\hab}{H_{\mathrm{ab}}}\newcommand{\htor}{H_{\mathrm{tor}}}
\newcommand{\htot}{H_{\mathrm{sa}}}
\DeclareMathOperator{\GSp}{GSp}

\title[Torsion and small points in semiabelian schemes]{Torsion, Betti rank, and small points in semiabelian schemes}
\author{Khai-Hoan Nguyen-Dang}
\address{Morningside Center of Mathematics, Chinese Academy of Sciences, Beijing, China}
\email{khaihoann@gmail.com}
\date{September 7, 2026}
\hypersetup{pdftitle={Torsion, Betti rank, and small points in semiabelian schemes},
 pdfauthor={Khai-Hoan Nguyen-Dang},pdfsubject={Relative torsion, small canonical heights, and semiabelian Betti geometry}}
\begin{document}
\begin{abstract}
Let $\G\to S$ be a semiabelian scheme over a smooth complex algebraic
base, with abelian relative dimension $g$ and toric rank $r$.
We prove that a dominating irreducible subvariety not contained in a proper
relative special subvariety has Zariski-dense fibrewise torsion if and
only if its mixed Betti rank is $2(g+r)$.  In particular, its dimension
is at least $g+r$.  For every prime, primary torsion is then dense even
among points of maximal rank.  The exceptional loci include generalized
Ribet loci.  Over algebraic numbers, we prove that generic small points
force full mixed rank whenever the abelian Betti rank is maximal.
This yields a transfer from abelian relative Bogomolov dimension theorems
to arbitrary toric extensions, together with toric height gaps above
abelian torsion and order-linear gluing-height bounds for nodal Pell
equations.  The proof combines algebraic quotients of marked $1$-motives,
quadratic nef compensation, a central-monodromy calculation, and
algebraicity of mixed Betti strata.
\end{abstract}
\maketitle
\tableofcontents

\section{Introduction}\label{sec:introduction}

The Manin--Mumford theorem, proved by Raynaud, is an important and remarkable theorem in arithmetic geometry. It says that an irreducible subvariety
of an abelian variety in characteristic zero contains a Zariski-dense set of
torsion points precisely when it is a translate of an abelian subvariety by
a torsion point \cite{Raynaud1983}. After that Laurent proved the corresponding theorem
for tori, and Hindry established its semiabelian varieties
\cite{Laurent1984,Hindry1988}; McQuillan placed this rigidity in the broader
setting of division points \cite{McQuillan1995}.  The Bogomolov theorems of
Ullmo and Zhang strengthen the abelian statement arithmetically: outside a
proper closed subset of a subvariety that is not a torsion coset, the
canonical height is bounded away from zero \cite{Ullmo1998,Zhang1998}.

In a family, torsion is imposed separately in each fibre.  A section can
therefore have infinitely many torsion specializations without being a
torsion section, and the question is how these specializations constrain
its variation.  The work of Masser--Zannier on simultaneous torsion in
elliptic families gave a concrete form to this question
\cite{MasserZannier2008}.  Its connection with transcendental uniformization
is made especially effective by the Pila--Zannier method
\cite{PilaZannier2008}: logarithms expressed in a local period basis turn
torsion into a rationality condition.  The resulting Betti map, studied by
Andr\'e--Corvaja--Zannier and Gao, links the differential rank of a section
to algebraic quotients and unlikely intersections
\cite{AndreCorvajaZannier2020,Gao2020,GaoCorrigendum2021}.

For abelian schemes, Gao--Habegger have proved the relative
Manin--Mumford theorem over $\C$ \cite[Theorems~1.1 and~1.3]{GaoHabegger2026}.
If $\A\to S$ has relative dimension $g$ and an irreducible subvariety
$X\subseteq\A$ satisfies
\[
  \ol{\bigcup_{n\in\Z}[n]X}^{\Zar}=\A,
\]
then fibrewise torsion is Zariski dense in $X$ if and only if its generic
real Betti rank is $2g$.  In particular, density forces $\dim X\ge g$.
This theorem treats arbitrary-dimensional subvarieties and arbitrary
complex coefficients.  The issue considered here is the additional
geometry and arithmetic introduced by a toric extension.

\subsection{Relative Manin--Mumford theorem for semiabelian schemes}
Let
\[
  0\longrightarrow\mathcal T\longrightarrow\G
    \longrightarrow\A\longrightarrow0
\]
be a semiabelian scheme, with abelian relative dimension $g$, toric rank
$r$, and total relative dimension $d=g+r$.  After a finite torus-splitting
refinement, its extension class is given by sections
$q_1,\ldots,q_r$ of $\A^\vee$.  A lift of an abelian point $p$ consists of
nonzero elements of the Poincar\'e fibres $\PP_{\A}(p,q_i)$.  Thus the
multiplicative coordinates depend on both $p$ and the extension
parameters: the problem does not separate into an abelian problem and an
independent torus problem \cite{BertrandEdixhoven2020,GuHuang2024}.

This dependence also changes the exceptional loci.  A Ribet section may
generate the geometric generic semiabelian fibre and nevertheless have
dense torsion specializations over a base of smaller dimension than that
fibre \cite{BMPZ2016,BertrandEdixhoven2020}.  The appropriate exceptions
are consequently not just torsion translates of subgroup schemes.  We use
the relative special subvarieties of Gu--Huang: after finite level and
polarization refinements, these are components of pullbacks of special
subvarieties of the universal semiabelian mixed Shimura variety
\cite[Definition~6.13]{GuHuang2024}.  They include generalized Ribet loci.
Definition~\ref{def:special} gives the precise convention, and
Example~\ref{ex:ribet} shows why generic cyclic generation alone is
insufficient.  This formulation is compatible with Pink's
unlikely-intersection framework \cite{Pink2005}, without assuming its
conjectural arithmetic statements.

We write $\G_{\tors}$ for the set of fibrewise torsion points, and call a
dominating subvariety \emph{relatively nonspecial} if it is not contained
in a proper relative special subvariety.  Properness is relative to the
marked-point fibre: a condition on the abelian moduli alone need not give
a proper locus in the family.  Our first main result is the following.

\Needspace{11\baselineskip}
\begin{theorem}[Relative Manin--Mumford]\label{thm:main}
Let $S$ be a smooth irreducible complex quasi-projective variety, let
$\pi:\G\to S$ be a semiabelian scheme of relative dimension $d=g+r$,
and let $X\subseteq\G$ be an irreducible closed subvariety dominating $S$.
Suppose that
\[
  \ol{X(\C)\cap\G_{\tors}}^{\Zar}=X.
\]
Then either $\dim X\ge d$, or $X$ is contained in a proper special
subvariety of $\G$ in the sense of Definition~\ref{def:special}.
\end{theorem}

Theorem~\ref{thm:main} establishes the dimension statement formulated
in Gu--Huang's revised relative Manin--Mumford conjecture
\cite[Conjecture~6.16]{GuHuang2024}, with proper special containment
understood in the pullback sense of Definition~\ref{def:special}.
It extends the abelian dimension bound to arbitrary toric extensions
while retaining the generalized Ribet loci among the exceptions.
The surface results of Bertrand--Masser--Pillay--Zannier and
Bertrand--Schmidt address important earlier cases
\cite{BMPZ2016,BertrandSchmidt2019}; the argument here does not reduce a
higher-dimensional base to a curve carrying a dense torsion set.  Instead,
it detects every relative toric direction through the central monodromy
of a marked $1$-motive.

To state the accompanying rank criterion, choose local period coordinates
$(\tau,z,w,t)$ and write
\[
  z=a+\tau b,\qquad a,b\in\R^g,\qquad c=t-wb\in\C^r.
\]
The mixed Betti map is $\beta=(a,b,c)$, with target
$\R^{2g}\times\C^r$.  In the integral Tate normalization, torsion means
that $a,b,c$ are real and rational modulo the period lattice.  Its fibres
are complex leaves, and its real differential rank is even.  We write
$\ell(X)=\tfrac12\max\rank_\R d\beta$ on the smooth locus, so that
$\ell(X)\le\min\{\dim X,g+r\}$.

\begin{corollary}[Torsion and mixed Betti rank]\label{cor:main-rank}
Let $S,\G,X$ be as in Theorem~\ref{thm:main}, without initially assuming
torsion density, and suppose that $X$ is relatively nonspecial.
The following conditions are equivalent:
\begin{enumerate}
\item $X(\C)\cap\G_{\tors}$ is Zariski dense in $X$;
\item the mixed Betti map has generic real rank $2(g+r)$;
\item for every rational prime $\ell_0$, the set
$X(\C)\cap\G[\ell_0^\infty]$ is Zariski dense in $X$.
\end{enumerate}
Here $\G[\ell_0^\infty]=\bigcup_{j\ge0}\G[\ell_0^j]$ is read
set-theoretically.  Density for some fixed prime is another equivalent
condition.  When these conditions hold, the primary torsion points of
maximal mixed rank are themselves Zariski dense, for every prime.
\end{corollary}

The converse from full rank to primary-torsion density holds without
relative nonspecialness.  Its proof in
Proposition~\ref{prop:primary-density} is algebraic after one use of the
abelian Betti submersion: on a component where $[\ell_0^j]p=0$, the
multiple $[\ell_0^j]h$ maps dominantly to the torus, and primary roots of
unity pull back Zariski densely.  This avoids an incorrect appeal to
density of torsion in the whole mixed Betti target.

For a simple model, take a fixed elliptic curve $E$, the base
$B=E\times\Gm$, the constant family $(E\times\Gm)\times B$, and the
section $h(x,\zeta)=(x,\zeta)$.  Locally
$x=a+\tau b$ and $\zeta=e^{2\pi i c}$, so $\beta=(a,b,c)$ has real rank
four.  Torsion is $E_{\tors}\times\mu_\infty$: it is Zariski dense in
$B$, but analytically dense only in $E(\C)\times\{ |\zeta|=1\}$.
There are three real torsion coordinates, although the full Betti target
has dimension four.  In general the corresponding numbers are $2g+r$
and $2g+2r$.  This distinction also explains the exponent $2g+r$ in the
local exact-order counts of Corollary~\ref{cor:local-count}.  Those counts
hold at transverse unit-norm points; they do not assert a global degree
formula or a uniform threshold for the occurrence of exact orders.

For singular bases we apply the theorem on a smooth dense open and take
the closure of the containing special locus.  For a subvariety not
dominating the given base, we first replace the base by the normalization
of the closure of its image.  The theorem concerns torsion density, not
the more general subgroup-intersection assertion of
\cite[Conjecture~6.15]{GuHuang2024}.

\subsection{Small points and arithmetic consequences}
Over $\Qbar$, canonical heights measure a second form of arithmetic
abundance.  Fix a torus-splitting refinement and a symmetric relatively
ample polarization $\Theta$ of the abelian quotient.  For a point $x$
with abelian projection $p(x)$, put
\begin{equation}\label{eq:intro-height}
 \begin{gathered}
  \htot(x)=\hab(x)+\htor(x),\qquad
  \hab(x)=\widehat h_\Theta(p(x)),\\
  \htor(x)=\sum_v n_v\sum_{i=1}^r|u_{i,v}(x)|.
 \end{gathered}
\end{equation}
Here $u_{i,v}$ is the negative logarithm of the canonically rigidified
Poincar\'e norm in the $i$th toric pushout, and $n_v$ is the normalized
local degree.  Section~\ref{sec:canonical-height} constructs this height
and proves its exact functoriality.  It is nonnegative and vanishes
precisely at fibrewise torsion points.  On a split product it is the
N\'eron--Tate height plus twice the sum of the coordinate Weil heights.
A sequence is called \emph{generic} if each proper closed subvariety
contains only finitely many of its terms.

The main arithmetic result supplies the extension directions once the
abelian Betti map has maximal rank.

\begin{theorem}[Small points and mixed rank]\label{thm:smallrank}
Let $B/\Qbar$ be smooth, irreducible, and quasi-projective, let $\G/B$ be
a semiabelian scheme of relative dimension $g+r$, and let $h$ be a
relatively nonspecial algebraic section, with abelian projection $p$.
Assume $\rank_\R db_p=2g$.  If there is a generic sequence
$b_n\in B(\Qbar)$ with $\htot(h(b_n))\to0$, then
\[
  \rank_\R d\beta=2(g+r),\qquad \dim B\ge g+r.
\]
Equivalently, under the same nonspecialness and abelian-rank hypotheses,
if the mixed rank is less than $2(g+r)$, there are $\eta>0$ and a proper
closed $Z\subsetneq B$ such that
$\htot(h(b))\ge\eta$ for every $b\in B(\Qbar)\setminus Z$.
The analogous statement holds for a dominating subvariety with its
tautological marked point.
\end{theorem}

The hypothesis on the abelian rank is essential to the scope of this
statement.  Theorem~\ref{thm:bogomolov-transfer} gives a precise way to
supply it: an abelian Bogomolov dimension theorem for a class stable under
isogenies and quotients implies the corresponding small-points and
mixed-rank criterion for arbitrary semiabelian extensions of that class.
The proof first applies a normal quotient to any deficient abelian rank,
and then invokes Theorem~\ref{thm:smallrank}.  K\"uhne's theorem for
products of elliptic families and the isotrivial result of
Dimitrov--Gao--Habegger supply the two unconditional classes treated in
Corollaries~\ref{cor:elliptic-bogomolov} and~\ref{cor:isotrivial-bogomolov}
\cite[Theorem~1]{KuehneRBC2023}, \cite[Proposition~4.1]{DGH2022}.
The semiabelian extension may be nonsplit and may vary even when its
abelian quotient is constant.  No general abelian relative Bogomolov
conjecture is assumed in either of these two cases, or proved here.

For an arbitrary abelian quotient, exact torsion in that quotient gives
a useful substitute for an abelian small-points theorem.
Theorem~\ref{thm:toric-gap} proves that if $X$ is relatively nonspecial
and $\ell(X)<g+r$, then outside a proper closed subset
\[
  p(x)\text{ torsion}\quad\Longrightarrow\quad\htor(x)\ge\eta
\]
for some $\eta>0$.  There is no dimension restriction on $X$ beyond the
rank deficiency.  If $p(x)$ has order $N$ and
$[N]x=(t_1,\ldots,t_r)$ in the split torus, exact linear scaling gives
\[
  2\sum_i h_{\mathrm{Weil}}(t_i)=N\htor(x)\ge\eta N.
\]
For a nodal curve, the $t_i$ are ratios of values of a rational function
at the pairs of points identified to form the nodes.  The Pell
application in Section~\ref{sec:pell} therefore measures an obstruction
invisible in the ordinary Jacobian: the function exists on the
normalization, but its gluing ratios have heights growing at least
linearly with its divisor order.

These are statements about specializations in moving families.  The
absolute semiabelian Bogomolov theorem is due to David--Philippon;
K\"uhne later proved strong equidistribution and gave another proof
\cite{DavidPhilippon2000,KuehneSmall2022}, following earlier work of
Chambert-Loir in the almost-split case
\cite{ChambertLoir1999,ChambertLoir2000}.
The explicit fixed generalized Jacobian in
Proposition~\ref{prop:explicit-pell} also falls under these absolute
results.  Its role here is to verify the relative hypotheses in concrete
polynomials and exhibit the difference between ordinary and nodal Pell
equations, not to claim a new absolute Bogomolov case.

\subsection{Ingredients and method of the proof}
We pass from $X$ to a smooth marked classifying image $B$ in a universal
semiabelian family and use its tautological section $h$.  This removes
parameters along which the marked family does not change.  The rational
realization $E$ of $[\Z\xrightarrow h\G_B]$ has weights $0,-1,-2$.
Write $P$ for its generic Mumford--Tate group, $M$ for connected algebraic
monodromy, and $U\simeq\Ga^r$ for the central Tate translations.  Relative
nonspecialness says that $P$ contains all marked-point translations
$K=\Hom(\Q e_0,W_{-1}E)$.

Two estimates account for the dimension bound.  Maximal abelian Betti
rank and the mixed fixed-part theorem give full abelian relative
monodromy.  Normality and the Poincar\'e Heisenberg bracket then show
that a missing central direction would produce a constant multiplicative
quotient whose marked value is not a root of unity.  A torsion value, or
a generic sequence of toric height tending to zero, excludes that
quotient.  Thus
\begin{equation}\label{eq:strategy-full-center}
  M\cap U=U.
\end{equation}
The arithmetic argument gives the opposing estimate, in the immersive
case $\ell(B)=m:=\dim B$,
\begin{equation}\label{eq:strategy-center-bound}
  \dim(M\cap U)\le m-g.
\end{equation}
Together these imply $m\ge g+r$.

The quotient construction makes the immersive reduction algebraic.
For a connected rational normal subgroup $N\lhd P$, we take
$W_N=(N-1)E$, the smallest subspace on whose quotient $N$ acts trivially.
Its saturation includes the images of the extension parameters, not
just the visible marked-point translations.  We algebraize $W_N$ as a
semiabelian subgroup over the original algebraically closed field.
Weak mixed Ax--Schanuel then turns a positive-dimensional Betti leaf
into a quotient problem of smaller base dimension, while the quotient
homomorphism preserves torsion and sequences of canonical height tending
to zero \cite{GaoKlingler2024,Chiu2025}.
The same argument proves the normal-quotient rank formula
\[
 \ell(B)=\min_{N\lhd P}
 \left(\dim\ol{q_N(B)}^{\Zar}
       +\tfrac12\dim_\Q V_{p,N}+\dim(N\cap U)\right)
\]
of Corollary~\ref{cor:normal-quotient-rank-formula}.  Here $q_N$ is the
normal mixed Shimura quotient and $V_{p,N}$ is the marked abelian
projection of its weight-$-1$ Lie algebra.  The final term records the
weight-$-2$ directions absent from the weight-$-1$ normal-function
formula of Gao--Zhang \cite[Theorem~1.2]{GaoZhang2026}.

A one-character calculation explains the arithmetic construction.
Write $q=\lambda a$ under a principal polarization, and set
$\bL=p^*\bTheta$, $\bQ=a^*\bTheta$.  The section $h$ trivializes the
Poincar\'e pullback, giving an adelic divisor $\bD$ with Green function
$u=-\log\|h\|$.  The theta identity gives, for each positive integer $n$,
\[
 \frac1{n^2}(np\pm a)^*\bTheta
   =\bL\pm\frac1n\bD+\frac1{n^2}\bQ.
\]
Both sides are compared using this fixed trivialization on $B$.
Taking the maximum of the two nef bundles on dominating models yields
\[
 \bL+\frac1n\max(\bD,-\bD)+\frac1{n^2}\bQ\quad\text{nef}.
\]
Thus the absolute logarithmic norm enters to first order, whereas the
compensating extension-parameter height enters only to second order.
Lemma~\ref{lem:compensation} carries out this calculation simultaneously
for any finite collection of characters, with one compensation constant
as their weights vary.

The resulting arithmetic intersection polynomial forces all relevant
mixed Chern measures to be supported where every logarithmic norm
vanishes.  A positive leading intersection first gives bounded
extension-parameter height; if that intersection vanishes, positivity
and comparison of character divisors make the measures themselves zero.
Both cases are needed.  The arithmetic framework and the total-mass
identity are those of Yuan--Zhang \cite{YuanZhang2026}; Guo proves a
more general local total-mass theorem \cite{Guo2025}.

On an abelian Betti leaf, the logarithmic norms are the real parts of a
holomorphic vector $F=-2\pi i(t-wb)$.  The all-character support condition
forces this real vector into the leafwise complex tangent image.
We polarize the resulting minors, hold the antiholomorphic period data
fixed, and continue the holomorphic first jet around monodromy.  A
central translation changes only the additional radial column.  Subtracting
the two column conditions gives
\[
  (M\cap U)_\C\subseteq D(\ker A),\qquad
  \dim(M\cap U)\le\ell(B)-g,
\]
where $A=dz-d\tau\,b$ and $D=dt-dw\,b$.
Proposition~\ref{prop:jet-bound} isolates this geometric implication.
The separation of holomorphic and antiholomorphic variables follows the
method of \cite[\S5.2]{AndreCorvajaZannier2020} and
\cite[\S9]{KuehneRBC2023}; its use here is the central-translation
calculation, not a new polarization principle or an Ax--Schanuel theorem
with derivatives.

Finally, passage from algebraic numbers to arbitrary complex
coefficients requires algebraicity of the mixed Betti strata.
We place the leaf equations in one algebraic family in the period
torsor and apply the flat-leaf intersection results of Baldi--Urbanik
\cite[Propositions~3.10 and~8.1]{BaldiUrbanik2025}, whose
principal-bundle input is developed in \cite{BCFN2026}.
Subtracting the dimension of the period-frame stabilizer converts their
torsor inequality into finitely many closed normal-quotient fibre loci.
Proposition~\ref{prop:betti-strata} identifies these with the Betti strata.
The resulting transcendence-degree induction follows the abelian
strategy of \cite[\S10]{GaoHabegger2026}; compare also
\cite[Appendix~A]{CorvajaTsimermanZannier2026}.
The weight-$-2$ orbit calculation is supplied here, rather than inferred
from the weight-$-1$ strata theorem of \cite{GaoZhang2026}.

\subsection{Organization}

Section~\ref{sec:universal} develops the special geometry and Betti map,
including the primary-torsion converse and local counts.
Section~\ref{sec:quotients} constructs the quotients and proves the
monodromy statements.  Sections~\ref{sec:adelic}
and~\ref{sec:central-bound} establish arithmetic support and the
small-points theorem.  Section~\ref{sec:strata} proves the strata theorem
and completes the complex torsion theorem.  The final two sections give
the Bogomolov transfer, toric height gaps, and Pell applications.
Throughout, a variety is reduced and of finite type over its stated field;
we use smooth dense opens when passing to period maps.

\section{Special geometry and Betti maps}\label{sec:universal}
\subsection{Universal families and marked $1$-motives}\label{sec:examples}
Choose a neat level and, when needed, a symmetric theta bundle after a further
finite refinement.  We denote the universal principally polarized abelian
scheme by $\mathscr U\to\mathscr A_g$.  For $r\ge0$ put
\[
 \BB_{g,r}=(\mathscr U^\vee)^r_{/\mathscr A_g}.
\]
The universal semiabelian family
\begin{equation}\label{eq:universal-family}
 \MM_{g,r}\longrightarrow\BB_{g,r}
\end{equation}
is the fibre product of the $r$ Poincar\'e torsors over the \emph{common}
abelian-point coordinate.  Its points are written
\[
 (A,p,q_1,\ldots,q_r,\xi_1,\ldots,\xi_r),
 \qquad \xi_i\in\PP_A(p,q_i)^\times.
\]
In particular, this is not a product of $r$ independent abelian-point coordinates.
The total space is a mixed Shimura variety, at the chosen level
\cite{Pink1990,BertrandEdixhoven2020,GuHuang2024}.

The exceptional loci must be defined relative to the family, rather than only inside an individual fibre.

\begin{definition}\label{def:special}
An irreducible subvariety $Z\subseteq\G$ dominating a normal base $S$ is called
special if, after a finite \'etale refinement supplying a torus basis,
polarization data and level, it is an irreducible component of the pullback of
a special subvariety of the appropriate universal semiabelian mixed Shimura
variety.  We use finite isogeny correspondences to compare polarization types.
This is the pullback convention of \cite[Definition~6.13]{GuHuang2024}.

An irreducible subvariety is \emph{relatively nonspecial} if it is not contained
in any proper special subvariety dominating its base.
\end{definition}

Multiplication by $n$ on a semiabelian scheme in characteristic zero is finite
\'etale and surjective.  A finite isogeny therefore preserves and reflects
fibrewise torsion: if its kernel is killed by $e$, then an $n$-torsion image has
an $en$-torsion preimage.  Finite \'etale changes of base preserve the torsion
condition.  They also preserve density on every dominating component, since
finite \'etale maps are open.  After an arbitrary finite generically \'etale
cover, we first remove its branch locus.

A torus scheme becomes split after a finite \'etale cover: its character lattice
is a locally constant integral lattice, and the continuous image of the
profinite fundamental group in the discrete group $\mathrm{GL}_r(\Z)$ is finite.
For the polarization reduction, choose an isogeny $\alpha:A\to A'$ to a
principally polarized abelian variety at the generic point.  After finite
extension, lift each extension parameter through
$\alpha^\vee:(A')^\vee\to A^\vee$.  If the lifted extension is $G'$, the
identity of extension classes identifies $G$ with $\alpha^*G'$, and gives a
finite isogeny $G\to G'$.  All these data spread over a nonempty open of the
base.  This explains the direction of the isogeny on the semiabelian family;
one does not push an arbitrary extension through an abelian isogeny without
first lifting its extension parameters.

Images of special subvarieties, and irreducible components of inverse images
under finite level maps or isogeny correspondences, are special.  This follows
from the description by mixed Shimura subdata and rational Hecke translates
\cite{Pink1990}.  Equivalently, the rational homology map defining the isogeny
is a Hodge tensor over the common special closure; after clearing denominators
and refining the level, its graph is a finite mixed Shimura correspondence.
Such maps preserve dimensions, so properness in the geometric generic fibre
is preserved as well.

Proper special containment is compatible with the refinements just described.  We record the descent statement for later use.

\begin{lemma}\label{lem:descent-special}
Suppose a section or a dominating subvariety acquires proper special containment
after one of the finite refinements above, or after passage to its tautological
section.  Then the original subvariety has proper special containment.
\end{lemma}
\begin{proof}
Choose a common level.  A containing special locus is a component of the pullback
of a universal special subvariety $T$.  The original subvariety is contained in
the pullback of the image of $T$ under the finite level or isogeny
correspondence.  An irreducible component containing it is special.

For a tautological base change, the same universal $T$ works before the base
change.  If its pullback before the change were the whole family, its pullback
after the change would also be the whole family, which is irreducible over an
irreducible base.  This contradicts the existence of a proper containing
component.  For finite correspondences, properness follows from equality of
source and target relative dimensions and finiteness.  Closures across the
removed open set are taken in the same universal special pullback.
\end{proof}

Passing from a subvariety to its marked classifying image removes irrelevant parameter directions without losing torsion density.

\begin{proposition}\label{prop:universal-reduction}
To prove \cref{thm:main}, it is enough to treat the following situation.  The
base $B$ is a smooth irreducible locally closed subvariety of $\MM_{g,r}$, the
family over $B$ is pulled back from \eqref{eq:universal-family}, and $h$ is its
tautological section.  Its torsion values are dense.  If the original
subvariety is relatively nonspecial and has dimension less than $d$, then the
new problem is relatively nonspecial and $m:=\dim B<d$.
\end{proposition}
\begin{proof}
Apply the preceding refinements and the classifying map to $X$.  Let $B$ be a
smooth dense open of the irreducible closure of its image.  The family is the
pullback of the universal family and the image point defines $h$.  Images of
dense torsion sets are dense in the image closure, and $\dim B\le\dim X$.
If $h$ had proper special containment, \cref{lem:descent-special} would provide
proper special containment for $X$.  Restricting to the smooth open does not
change density.
\end{proof}

Let $E$ be the rational homological realization of $[\Z\xrightarrow h\G_B]$.
It is an admissible graded-polarizable variation of mixed Hodge structures,
with
\[
 \Gr^W_0E=\Q(0),\qquad \Gr^W_{-1}E=V=H_1(\A/B,\Q),
 \qquad \Gr^W_{-2}E=\TT=\Q(1)^r.
\]
We use admissibility on the whole smooth base, including points where
$h$ meets the identity.  The relative exponential sequence gives a
variation on all of $B$ whose fibres are Deligne's realizations
\cite[Construction~10.1.3]{Deligne1974}.  Here is the extension argument
that makes the admissibility assertion precise.  If $h$ is not identically
zero, on the dense open $B^\circ$ where $0$ and $h$ are disjoint it is
identified with
\[
 H_1\bigl(\G_{B^\circ}/B^\circ,\{0,h\};\Q\bigr).
\]
This is a geometric relative-cohomology variation, in homological
notation, and is admissible; geometric admissibility is provided by
\cite{ElZein1986}, also used in
\cite[\S3.1, footnote~3]{GaoZhang2026}.
Choose a smooth compactification $\overline B$ with normal-crossing
boundary.  Proposition~1.10.1 of \cite[p.~995]{Kashiwara1986} applies with
$X=\overline B$, $X^*=B$ and $Y$ the closure of $B\setminus B^\circ$:
a variation already defined on $B$ is admissible relative to
$\overline B$ if its restriction to that dense open is admissible
relative to $\overline B$.  Thus admissibility extends across the
collision locus.  If $h=0$ identically, the realization is
$\Q(0)\oplus H_1(\G_B/B,\Q)$ and geometric admissibility applies directly.
In particular, restriction to every test disc in the compactification
has quasi-unipotent monodromy, a relative monodromy filtration, and the
locally free extended Hodge-graded pieces specified in
\cite[\S\S1.8--1.9]{Kashiwara1986}.
For smooth $1$-motives over algebraic curves these properties are also
established directly in \cite[Lemma~5]{Andre1992}.

The weight-graded pieces are polarized, and all subquotients below are
taken in the admissible category.  The mixed fixed-part theorem and
normality are used in the form of
\cite[Theorem~1 and its proof]{Andre1992}; see also
\cite[Theorem~4.19]{SteenbrinkZucker1985}.
After a finite \'etale cover we arrange that algebraic monodromy is
connected.  All subsequent constructions may be restricted to a common
smooth dense open.

With a rational weight splitting, the ambient group is represented by
\begin{equation}\label{eq:matrix-group}
 \begin{pmatrix}
 \mu(a)I_r&q&t\\0&a&p\\0&0&1
 \end{pmatrix},
 \qquad a\in\GSp(V),\quad q\in\Hom(V,\TT),\quad
 p\in V,\quad t\in\TT.
\end{equation}
When $g=0<r$, the pure factor is interpreted as the Tate group
$\Gm$, acting by its scalar character on $\TT$, and all abelian
blocks are empty.  The case $g=r=0$ is trivial and is omitted from
these matrix conventions.  Its unipotent radical has bracket
\begin{equation}\label{eq:heisenberg}
 [(p,q),(p',q')]=q(p')-q'(p)\in\TT.
\end{equation}
The kernel of the map forgetting $\Gr^W_0E$ is the vector group
\[
 K=\Hom(\Q e_0,W_{-1}E)=V_p\oplus U,
 \qquad U=\TT\otimes\Ga.
\]
The group $U$ is central in the unipotent radical of the ambient
matrix group, and hence in the relevant unipotent radicals below.
It need not be central in the entire Mumford--Tate group: the pure
similitude factor acts on it through its Tate character.  The
splitting $K=V_p\oplus U$ and the displayed projections are always
read in compatible rational weight splittings.
We identify additive rational vector groups with their rational Lie
algebras when writing weight-graded subspaces.
Write $P=\MT(E)$ and $P_0=\MT(W_{-1}E)$ for the generic groups, in a common
rational fibre.  The map $P\to P_0$ is surjective by the tensor description of
Mumford--Tate groups.

The following criterion translates the relative exceptional-locus condition into a condition on the generic Mumford--Tate group.

\begin{lemma}\label{lem:full-relative}
In the universal-section setting, relative nonspecialness is equivalent to
\begin{equation}\label{eq:full-relative}
 K\subseteq P,
 \qquad\text{or, equivalently,}\qquad
 P=\operatorname{preimage}(P_0).
\end{equation}
It implies that $p_{\bar\eta}$ generates $\A_{\bar\eta}$.
\end{lemma}
\begin{proof}
The special closures of the section image and of its classifying base are the
mixed Shimura loci attached to $P$ and $P_0$.  Their generic relative domain
fibre is the orbit of $P\cap K$.  A proper rational Hodge subgroup of $K$ has
complex fibre dimension strictly less than $g+r$: its weight-$-1$ dimension
contributes half its rational dimension, and its weight-$-2$ dimension
contributes its full dimension.  Thus $P\cap K\ne K$ gives proper relative
special containment.  Conversely, if $K\subseteq P$, the special closure
contains the whole relative fibre over the special closure of the classifying
base.  A proper special pullback containing the section would contradict this
full fibre.  This proves the equivalence.

If a multiple of $p_{\bar\eta}$ lay in a proper abelian subvariety, that
subvariety and the containment would spread after a finite cover.  The
quotient marked-point $1$-motive would have torsion marked point.  Its
Mumford--Tate group would consequently have zero relative translation kernel,
contrary to the image of the full $V_p$.  Equivalently, this is the proper
subgroup pullback in the abelian quotient.
\end{proof}

It follows from Gao--Habegger that, for a relatively nonspecial torsion-dense
section,
\begin{equation}\label{eq:abelian-rank}
 m\ge g,\qquad \rank_{\R}db_p=2g.
\end{equation}
Here $b_p$ is the local abelian Betti map.  Generic cyclic density of $p$ implies
$\ol{\Z p(B)}^{\Zar}=\A_B$, so the generation hypothesis in their theorem is
satisfied.  Generic cyclic density of $h$ alone must not be substituted for
\eqref{eq:full-relative}.

The specialness criterion can be tested on the following Ribet example.
It also shows why extra parameters and generic cyclic generation must be
kept separate from Betti rank.

\begin{example}\label{ex:ribet}
Let $E/\Qbar$ be an elliptic curve without complex multiplication, let
$A=E^2$ with its product principal polarization, and put
\[
 J=\begin{pmatrix}0&1\\-1&0\end{pmatrix}.
\]
Then $J^\dagger=-J$.  Take $B=A$, let $q:B\to A^\vee$ be the identity under
the polarization, and let $\HH_q/B$ be the Poincar\'e extension of
$A\times B$ by $\Gm$.  The normalized Ribet section for $J$ has projection
$p=2Jq$ and has torsion value at every torsion value of $q$; for $q\in A[n]$
its order divides $n^2$.  These facts follow from
\cite[Propositions~3.1 and~3.3]{BertrandEdixhoven2020}.

The two generic elliptic coordinates of $p$ satisfy no nonzero integral
relation, so $p$ is cyclically dense in $A$ over the geometric generic field.
The extension class $q$ is nontorsion.  A proper connected subgroup of
$\HH_q$ surjecting onto $A$ would have finite toric kernel and would split the
extension after an abelian isogeny, forcing $q$ to be torsion.  Thus the
Ribet section is generically cyclically dense in the whole three-dimensional
semiabelian fibre.

Nevertheless, the base has dimension two and the fibre dimension is three.
The section is contained in a proper generalized Ribet special locus, in
accordance with \cite{GuHuang2024}.  This shows why an ambient dimension-two
rigidity statement over curves cannot be extended to higher-dimensional bases.
\end{example}

\begin{example}
Pull back the preceding example to $B\times C$ for an arbitrary positive-
dimensional projective variety $C$, without changing the section or the family
in the $C$-direction.  Torsion remains dense, but the extra base directions
contribute no Betti rank.  A dimension argument on an arbitrary parameter
space must therefore either pass to the classifying image or control these
rank-deficient directions by a quotient.  Merely cutting a general curve
through a dense torsion set does not guarantee infinitely many torsion points
on that curve.
\end{example}

\begin{remark}
Gu--Huang classify the relevant special loci and deduce their revised relative
statement from mixed arithmetic Zilber--Pink.  No such conjectural implication
is used as a theorem in this paper.  The separate geometric input
used in Section~\ref{sec:strata} concerns flat-leaf intersections in an
algebraic period torsor carrying an admissible variation.  Conflating these two forms of Zilber--Pink would make
the claimed unconditional argument circular.
\end{remark}

\subsection{Period coordinates and normal orbits}\label{sec:betti}
On a simply connected analytic chart, choose period coordinates
\[
 (\tau,z,w,t)\in\mathfrak H_g\times\C^g\times
 \operatorname{Mat}_{r,g}(\C)\times\C^r.
\]
The $i$th Poincar\'e coordinate is $e^{2\pi i t_i}$.  We use the rescaling of
the algebraic coordinates of \cite[Section~4]{BertrandEdixhoven2020} for which
both the top row and the top-right entry have been divided by $2\pi i$.
A fixed nonzero scalar change does not affect the ranks below.  However,
the assertion that torsion has \emph{real rational} labels uses precisely
the chosen integral Tate normalization; it is not invariant under an
arbitrary complex rescaling.  We retain this normalization throughout.
For later comparison with a principal polarization, if $z_{a_i}$ is a
column logarithm of $a_i$ and $q_i=\lambda a_i$, our dual row coordinate is
\begin{equation}\label{eq:dual-coordinate-dictionary}
  w_i=-z_{a_i}^{\mathsf t}.
\end{equation}
This is the principal case of the dual-coordinate convention in
\cite[(2.18)]{BurgosHolmesDeJong2018}.

There are unique real columns $a,b$ with
\[
 z=a+\tau b,\qquad b=(\operatorname{Im}\tau)^{-1}\operatorname{Im}z.
\]
Put $c=t-wb$.  The mixed Betti map has local labels
\begin{equation}\label{eq:betti-labels}
 \beta=(a,b,c)\in\R^{2g}\times\C^r.
\end{equation}
These labels need not be global single-valued functions; their fibres form a
well-defined foliation.  The leaf with labels $(a,b,c)$ has equations
\begin{equation}\label{eq:leaf}
 z=a+\tau b,\qquad t=c+wb.
\end{equation}
It is complex algebraic in this period-coordinate chart.  The integral lattice
acts by rational affine transformations on the labels.  Fibrewise torsion
means that the labels are rational and real, modulo that lattice.  In
particular, a fixed-order torsion locus is locally a finite union of leaves.
This follows either from the explicit lattice action or by applying
multiplication to the three graded pieces of the marked-point $1$-motive.

Let $B$ be a smooth subvariety of the universal total space, and write
\begin{equation}\label{eq:forms-AD}
 A=dz-d\tau\,b,\qquad D=dt-dw\,b.
\end{equation}
At a fixed point these forms are complex-linear on the holomorphic
tangent space; their coefficients vary real-analytically because
$b$ does.  As real differentials, $A=da+\tau db$ and $D=dc+wdb$.  Since
$\operatorname{Im}\tau$ is invertible,
\[
 A(v)=0\Longleftrightarrow da(v)=db(v)=0.
\]
Hence the real kernel of $d\beta$ is exactly the complex kernel of
$\binom{A}{D}$.  On a nonempty rank-regular analytic open we define
\begin{equation}\label{eq:mixed-rank}
 \ell(B)=\rank_\C\begin{pmatrix}A\\D\end{pmatrix}
 =\tfrac12\rank_\R d\beta.
\end{equation}
Thus $0\le\ell(B)\le\min\{\dim B,g+r\}$.  More precisely,
$\ell(B)$ is the maximum of the complex ranks in
\eqref{eq:mixed-rank}; the real-analytic minors show that this maximum
is attained on a nonempty analytic open and off a proper real-analytic
rank locus on the connected smooth base.  No Zariski openness of that
locus is assumed.

The foliation is preserved by
$P_{\mathrm{univ}}(\R)^+U(\C)$.  It is useful to record what is required here:
labels may change, but two points on the same leaf remain on the same leaf.
For a real $q$-translation,
\[
 w\longmapsto w+x_1\tau+x_2,\qquad
 t\longmapsto t+x_1z,
\]
and therefore $c\mapsto c+x_1a-x_2b$.  For a real $p$-translation,
\[
 z\longmapsto z+y_1-\tau y_2,\qquad
 t\longmapsto t-wy_2,
\]
so $a\mapsto a+y_1$, $b\mapsto b-y_2$, and $c$ is unchanged.
A complex central translation adds a constant to $c$.
For clarity, the pure symplectic invariance can be checked on generators.
A symmetric real matrix $B_0$ gives
\[
  (\tau',z',w',t')=(\tau+B_0,z,w,t),\qquad
  (a',b',c')=(a-B_0b,b,c).
\]
For $H\in\mathrm{GL}_g(\R)$, one has
\[
  (\tau',z',w',t')=(H\tau H^{\mathsf t},Hz,wH^{\mathsf t},t),
  \quad (a',b',c')=(Ha,H^{-\mathsf t}b,c).
\]
The inversion generator acts by
\[
  (\tau',z',w',t')=
  (-\tau^{-1},\tau^{-1}z,w\tau^{-1},t-w\tau^{-1}z).
\]
Here $(a',b')=(b,-a)$ and
$c'=t-w\tau^{-1}(a+\tau b)+w\tau^{-1}a=t-wb=c$.
Thus each generator maps a leaf to a leaf.  Positive similitudes rescale
the Tate coordinates.  These formulas agree with the period action in
\cite[Section~4]{BertrandEdixhoven2020}.  Real group invariance of the
foliation and preservation of the rational torsion-label lattice are
different statements; the latter uses the rational or integral subgroup.

Normality makes the mixed Betti rank uniform along the weakly special orbits that occur in the proof.

\begin{lemma}\label{lem:orbit-rank}
Let $(P_1,\DD_1)$ be a mixed Shimura subdatum of the universal datum, and let
$N\lhd P_1$ be a connected rational normal subgroup defining a weakly special
orbit $T$.  Denote by $V_{p,N}$ the $p$-projection of
$\Gr^W_{-1}\Lie N$, and put $U_N=N\cap U$.  The mixed Betti map has constant
complex rank
\begin{equation}\label{eq:orbit-rank}
 h_N=\tfrac12\dim_\Q V_{p,N}+\dim U_N
\end{equation}
on the orbit.  Consequently its intersections with leaves have complex
dimension $\dim T-h_N$.
\end{lemma}
\begin{proof}
Choose a rational Levi of $P_1$ and a compatible splitting of the
representation $E$.  Conjugating the embedding of the datum and the period
coordinates by the same rational unipotent element does not change the
foliation or any of the dimensions being computed.  The domain attached
to the Levi is a split reference locus in $\DD_1$; the remaining points
are obtained using $\Ru(P_1)(\R)U_1(\C)$
\cite[\S\S1.16--1.18 and 2.19]{Pink1990}.
At a point of this reference locus the extension coordinates are
$z=w=t=0$ in the compatible splitting.  Since $N$ is normal, its Lie
algebra is stable under the weight cocharacter of this Levi, so its
tangent contribution can be computed weight by weight.
At this reference point $b=0$, so the two differential blocks are
$dz$ and $dt$.  The weight-zero and $q$-only tangent directions map to
zero.  Projection of the weight-$-1$ Hodge structure is strict for
$F^\bullet$, and hence its $p$-image modulo $F^0$ is precisely
$V_{p,N,\C}/F^0V_{p,N,\C}$.  The weight-$-2$ tangent space maps
identically to the central block.  With the compatible splitting this gives
\begin{equation}\label{eq:split-tangent-image}
 \im_\C\left(\begin{pmatrix}A\\D\end{pmatrix}\bigg|_{T_xT}\right)
 \simeq (V_{p,N,\C}/F^0V_{p,N,\C})\oplus U_{N,\C}.
\end{equation}
The first summand has dimension $\frac12\dim_\Q V_{p,N}$, since its
Hodge types are $(-1,0)$ and $(0,-1)$; the second has dimension
$\dim U_N$.  This proves \eqref{eq:orbit-rank} at the reference point.

The group $P_1(\R)^+U_1(\C)$ is transitive on $\DD_1$, preserves the foliation,
and normalizes $N$.  It therefore transports the tangent computation to every
$N$-orbit through every point of the domain.  One may write these orbits as
$N(\R)^+\Ru(N)(\C)x$ or as $N(\R)^+U_N(\C)x$: the real weight-$-1$ space
already surjects onto its complex quotient by $F^0$.  This proves constancy
and the stated leaf dimension.
\end{proof}

The algebraic equations of a leaf turn the preceding tangent computation into a bound for algebraic closures.

\begin{corollary}\label{cor:leaf-closure}
Let $Z$ be an irreducible complex analytic germ contained in a weakly special
orbit $T$ and in a Betti leaf.  Its Zariski closure in algebraic period
coordinates satisfies
\begin{equation}\label{eq:lift-closure-bound}
 \dim Z^{\Zar}\le\dim T-h(T).
\end{equation}
For compatible weakly special orbits $T\subseteq T'$, one has
\begin{equation}\label{eq:leaf-dim-monotone}
 \dim T-h(T)\le\dim T'-h(T').
\end{equation}
\end{corollary}
\begin{proof}
The orbit $T$ is analytically open in the appropriate complex algebraic orbit
in the compact dual.  Intersect that algebraic orbit with the algebraic
leaf equations \eqref{eq:leaf}, and take the irreducible component containing
the germ $Z$.  Since it meets $T$, its dimension is the local dimension
$\dim T-h(T)$ computed in \cref{lem:orbit-rank}.  Its closure contains
$Z^{\Zar}$, proving \eqref{eq:lift-closure-bound}.  Possible components lying
entirely in the boundary are irrelevant to this argument.  Finally, a leaf
intersection in $T$ is contained in the corresponding leaf intersection in
$T'$, proving \eqref{eq:leaf-dim-monotone}.
\end{proof}

\subsection{Primary torsion and local counts}
The complex central coordinates prevent a direct analytic-density
argument in the entire Betti target.  Algebraicity gives a different
route: first specialize the abelian projection to primary torsion, and
then multiply into the torus.  This proves a converse for Zariski density
without an arithmetic input.

\begin{proposition}\label{prop:primary-density}
Let $B$ be a smooth irreducible complex quasi-projective variety, let
$0\to\mathcal T\to\G\to\A\to0$ be a semiabelian scheme over $B$ of
relative dimension $g+r$, and let $h$ be an algebraic section.  If its
mixed Betti map has real rank $2(g+r)$ at some point, then, for every
rational prime $\ell_0$, the set
\[
 \{b\in B(\C):h(b)\in\G_b[\ell_0^\infty],\ 
                   \rank_\R d\beta_b=2(g+r)\}
\]
is Zariski dense in $B$.  No relative nonspecialness assumption is
required.  If the data are defined over an algebraically closed
$F_0\subseteq\C$, the same density holds using $b\in B(F_0)$.
\end{proposition}
\begin{proof}
The case $g+r=0$ is immediate.  Otherwise split the torus by a finite
\'etale cover; both rank and density on a dominating component descend.
No isogeny is needed.  We use the actual integral
period lattice, so that primary orders are not altered by a
polarization change.  On a simply connected analytic open, write its
generators in logarithmic coordinates $(Z,T)$ as
\[
 (\Omega_1,C_1),\quad(\Omega_2,C_2),\quad(0,I_r).
\]
The first two blocks project to an integral basis of the abelian period
lattice, and the third is the integral Tate basis.  Choose the first
abelian block invertible and put
\[
 z=\Omega_1^{-1}Z,\quad t=T-C_1z,\quad
 \tau=\Omega_1^{-1}\Omega_2,\quad w=C_2-C_1\tau.
\]
The generators become $(I_g,0),(\tau,w),(0,I_r)$.  The imaginary part of
$\tau$ is invertible, and the formulas
$z=a+\tau b$, $c=t-wb$, $A=dz-d\tau b$, $D=dt-dw b$ apply.  A
principal polarization is unnecessary for these formulas.  Empty
abelian or toric blocks cause no change to the argument.

Choose an analytic open $\Omega$ on which the mixed rank is $2(g+r)$.
The two row blocks have sizes $g$ and $r$, so
\[
 \rank_\C A=g,\qquad \rank_\C(D|_{\ker A})=r.
\]
Thus the abelian Betti map $(a,b)$ is a real submersion.  Fix $\ell_0$.
For $n=\ell_0^j$, let
\[
 B_n=\{b\in B:[n]p(b)=0\},\qquad p=\pi_\A h.
\]
These are closed algebraic loci.  Since $\Z[1/\ell_0]^{2g}$ is dense
in $\R^{2g}$, their union meets $\Omega$ in an analytically dense set.

Take $x\in B_n\cap\Omega$ and let $Y$ be the irreducible component of
$(B_n)_{\mathrm{red}}$ through $x$.  Locally, its branch is given by
$z=a+\tau b$ for fixed $a,b\in n^{-1}\Z^g$.  The differential is $A$,
so the branch is smooth at $x$ and $T_xY=\ker A_x$.  On all of $Y$,
$[n]h$ takes values in the split torus and defines an algebraic map
\[
 f_n:Y\longrightarrow\Gm^r.
\]
Subtracting the abelian period $n(a+\tau b)$ shows that its local toric
logarithm is $n(t-wb)=nc$, modulo the integral toric periods.  Hence
\[
 f_n^*(d\log X_i)|_{T_xY}=2\pi i\,nD_i|_{T_xY}.
\]
It follows that $f_n$ is dominant.  We now retain full rank on an
algebraic open of $Y$, without assuming that the full mixed-rank locus
in $B$ is algebraic.  Along $B_n$ the sections $s_n=[n]p$ and the
identity section $e$ agree, so their vertical differential difference
is the algebraic bundle map
\begin{equation}\label{eq:torsion-vertical-differential}
 \delta_n=(ds_n-de)|_{B_n}:
 T_B|_{B_n}\longrightarrow\Lie(\A/B)|_{B_n}.
\end{equation}
It is vertical because both differentials induce the identity on
$T_B$.  In the period chart it is $nA$, up to an invertible change of
vertical frame.  On the open where its rank is $g$, the zero locus
$B_n$ is smooth of codimension $g$ and its tangent bundle is $\ker A$.
Intersect that open in $Y$ with the open where $df_n$ has rank $r$;
call the resulting nonempty algebraic open $Y^\circ$.  It contains $x$.
The logarithmic formula above gives, for every $y\in Y^\circ$,
\[
 \rank_\C A_y=g,\qquad
 \rank_\C(D_y|_{\ker A_y})=r.
\]
Thus every point of $Y^\circ$ has full mixed rank, and
$f_n|_{Y^\circ}$ is dominant.

The subgroup $\mu_{\ell_0^\infty}^r$ is Zariski dense in $\Gm^r$:
in one variable it is infinite, and the product assertion follows by
applying a Laurent polynomial successively in its variables.  Its
inverse image under $f_n|_{Y^\circ}$ is Zariski dense in $Y^\circ$,
because the image of every nonempty Zariski open of $Y^\circ$ contains
a nonempty Zariski open of the target.  At each point of this inverse
image, $h$ is killed by a power of $\ell_0$ and the mixed rank is full.

The Zariski closure of the full-rank primary-torsion specializations therefore
contains every component of every $B_{\ell_0^j}$ meeting $\Omega$.
It contains an analytically dense subset of $\Omega$ and is analytically
closed; thus it contains $\Omega$, and irreducibility gives the result.

If the family is defined over an algebraically closed subfield
$F_0\subseteq\C$, the same density holds for its $F_0$-points.
Indeed $B_n$, its irreducible components, the two algebraic rank
conditions defining $Y^\circ$, and the fibres of $f_n$ over primary
torus torsion are all defined over $F_0$ after the same finite splitting
cover.  The $F_0$-points of each nonempty such fibre are Zariski dense;
thus they give the same closure argument.
\end{proof}

On a fixed torsion locus, there is an intrinsic interpretation of the algebraic rank
conditions used above.  For a positive integer $n$, let
$Z_n=\{[n]h=e\}$ have its natural closed subscheme structure in $B$,
where $e$ is the identity section of $\G/B$.  The two sections agree
along $Z_n$, so their differential difference is an algebraic vertical
bundle map
\begin{equation}\label{eq:full-torsion-differential}
 \Delta_n=\bigl(d([n]h)-de\bigr)|_{Z_n}:
 T_B|_{Z_n}\longrightarrow\Lie(\G/B)|_{Z_n}.
\end{equation}
In an integral period chart at a point of $Z_n$, choose the fixed
lattice branch indexed by $k,l\in\Z^g$ and $m_0\in\Z^r$.  A logarithm
of $[n]h$ near the identity is
\[
 \bigl(nz-k-\tau l,\ nt-wl-m_0\bigr).
\]
At the torsion point the abelian Betti column satisfies $b=l/n$, so
the differential of this logarithm is
$n\binom AD$, up to an invertible holomorphic change of vertical frame.
A base-dependent change of frame adds no term at this point, because
the logarithm itself vanishes there.  Consequently
\begin{equation}\label{eq:torsion-transverse-open}
 Z_n^{\mathrm{tr}}
 =\{x\in Z_n:\rank\Delta_{n,x}=g+r\}
 =\{x\in Z_n:\rank_\R d\beta_x=2(g+r)\}
\end{equation}
is an algebraic open subset of $Z_n$.  The Jacobian criterion makes it
smooth of codimension $g+r$ in $B$ when nonempty, with tangent space
$\ker\Delta_{n,x}$.  In critical dimension its points are isolated
and reduced, with local length one.  Since $[n]$ is \'etale, this is
also the local intersection multiplicity of $h$ with the corresponding
torsion section.  The construction works over the field of definition
of the family.  It asserts algebraicity on each fixed-order locus,
not on the entire maximal-rank locus of $B$.

The preceding proposition supplies full-rank torsion points Zariski
densely, hence full-rank unit-norm points.  This is a global algebraic
assertion: such a point need not lie in a prescribed analytic
neighborhood.  For example, on $\Gm$ the identity coordinate has full
rank everywhere, whereas a small neighborhood of $2$ contains no
torsion.  At a specified full-rank unit-norm point we have the following
local statement.  Here and below $u=2\pi\operatorname{Im}c$, and we
retain the smooth marked-family setting of
Proposition~\ref{prop:primary-density}.

\begin{proposition}\label{prop:unit-density}
Let $\G/B$ and $h$ be a smooth marked semiabelian family as above.
If $\rank_\R d\beta_x=2(g+r)$ and $u(x)=0$, then in a sufficiently
small neighborhood of $x$ the locus $V=\{u=0\}$ is smooth of real
codimension $r$.  For every prime $\ell_0$, primary-torsion
specializations are analytically dense in $V$.  Moreover $V$ is a
local uniqueness set for holomorphic functions.
\end{proposition}
\begin{proof}
The map $\beta$ is a submersion near $x$.  Its inverse image of the
real central locus is therefore smooth, and its restriction to $V$
is a submersion onto an open subset of $\R^{2g+r}$.  The dense subset
$\Z[1/\ell_0]^{2g+r}$ consists of primary-torsion labels, proving the
density assertion.

On an abelian Betti plaque, $a,b$ are fixed and $c=t-wb$ is holomorphic
of rank $r$, since $D|_{\ker A}$ has that rank.  Choose holomorphic
plaque coordinates $(c,\zeta)$.  The condition defining $V$ is
$c\in\R^r$, with $\zeta$ unrestricted.  A holomorphic function vanishing
there vanishes on the plaque by the identity theorem in each
$c$-variable.  The submersion theorem provides these real slices on all
nearby plaques.  The function consequently vanishes on a neighborhood
in the full base, as asserted.
\end{proof}

In critical dimension the integral Betti coordinates turn torsion into
a rectangular lattice problem.  M\"obius inversion separates exact orders
from orders dividing a given integer.  The normalization throughout is
the actual integral lattice, not an isogeny-normalized rational lattice.

\begin{corollary}\label{cor:local-count}
In the setting of Proposition~\ref{prop:unit-density}, assume
$\dim_\C B=g+r>0$ and put $s=2g+r$.  Choose an integral Betti chart
on which $\beta$ is a real-analytic diffeomorphism onto its image.
Order the real target coordinates as
$(a,b,\operatorname{Re}c,\operatorname{Im}c)$: the first $s$
coordinates are the real torsion labels and the last $r$ are the
imaginary central labels.  Choose bounded rectangular boxes
$I\subset\R^s$ and $J\subset\R^r$, with $0\in J$, such that
$I\times J$ is contained in the chart image.  Put
$\Omega_I=\beta^{-1}(I\times J)$.  Then
\begin{equation}\label{eq:local-dividing-count}
 A_I(n):=\#\{b\in\Omega_I:[n]h(b)=0\}
 =\operatorname{vol}(I)n^s+O_I(n^{s-1}).
\end{equation}
For exact orders, let
$E_I(n)=\#\{b\in\Omega_I:\operatorname{ord}(h(b))=n\}$.
Writing $\mu_{\mathrm{Mob}}$ for the M\"obius function and
\[
 J_s(n)=n^s\prod_{q\mid n,\ q\ \mathrm{prime}}(1-q^{-s}),
\]
one has
\begin{equation}\label{eq:local-exact-order-count}
 E_I(n)=\operatorname{vol}(I)J_s(n)
 +O_I\!\left(n^{s-1}\sum_{a\mid n}
              \frac{|\mu_{\mathrm{Mob}}(a)|}{a^{s-1}}\right).
\end{equation}
For $s\ge3$, the error in \eqref{eq:local-exact-order-count} is
$O_I(n^{s-1})$.  For $s=2$ it is $O_I(n\log(2n))$, and for $s=1$
it is $O_I(2^{\omega_{\mathrm{pr}}(n)})$, where
$\omega_{\mathrm{pr}}(n)$ counts the distinct prime divisors of $n$.
For each fixed prime $\ell_0$ and $j\ge1$,
\begin{equation}\label{eq:local-exact-primary-count}
 \begin{split}
 E_I(\ell_0^j)
 &=\operatorname{vol}(I)(1-\ell_0^{-s})\ell_0^{js}\\
 &\quad+O_{I,\ell_0}\bigl(\ell_0^{j(s-1)}\bigr).
 \end{split}
\end{equation}
If $\operatorname{vol}(I)>0$, every sufficiently large positive integer
occurs as an exact torsion order in $\Omega_I$, for every $s\ge1$.
All the counted points are simple transverse torsion intersections.
Such charts exist whenever $\dim_\C B=g+r>0$ and the mixed rank is full
somewhere; the threshold for exact orders depends on the chosen chart.
\end{corollary}
\begin{proof}
Full rank makes $\beta$ a local real-analytic diffeomorphism.  The
condition $[n]h=0$ is precisely $\operatorname{Im}c=0$ and
$(a,b,\operatorname{Re}c)\in n^{-1}\Z^s$.  Thus
$A_I(n)=\#(n^{-1}\Z^s\cap I)$.  In each interval factor the count is
its length times $n$, with bounded error.  Multiplying these estimates
gives \eqref{eq:local-dividing-count}, with one constant for all $n\ge1$.

Since $A_I(n)=\sum_{d\mid n}E_I(d)$, M\"obius inversion gives
\[
 E_I(n)=\sum_{a\mid n}\mu_{\mathrm{Mob}}(a)A_I(n/a).
\]
The main term is $\operatorname{vol}(I)\sum_{a\mid n}
\mu_{\mathrm{Mob}}(a)(n/a)^s=\operatorname{vol}(I)J_s(n)$.
The sum of the absolute errors is the error displayed in
\eqref{eq:local-exact-order-count}.  When $s\ge3$, the divisor sum is
at most $\zeta(s-1)$.  For $s=2$, it is at most
$\sum_{a=1}^n a^{-1}\le1+\log n$.  For $s=1$, it equals
$2^{\omega_{\mathrm{pr}}(n)}$.  These give the stated error bounds.
For a prime power the M\"obius sum has only the two terms
$a=1,\ell_0$, giving \eqref{eq:local-exact-primary-count}, including
when $s=1$.

Suppose now that $\operatorname{vol}(I)>0$.  For $s\ge2$, the bound
$J_s(n)\ge n^s/\zeta(s)$ makes the main term dominate the error.
For $s=1$, write $\varphi(n)=J_1(n)$ for Euler's totient.  The
prime-power factorization gives
\[
 \frac{\varphi(n)^2}{n}
   =\prod_{p^a\parallel n}p^{a-2}(p-1)^2\ge\frac12.
\]
Indeed every factor is at least one except possibly the factor for
$p=2$, $a=1$, which is $1/2$.  Since $2\le p^{1/4}$ for every
prime $p\ge17$, the six smaller primes give
$2^{\omega_{\mathrm{pr}}(n)}\le64n^{1/4}$.
Consequently $2^{\omega_{\mathrm{pr}}(n)}=o(\varphi(n))$, proving
eventual positivity of $E_I(n)$ also in rank one.

The entire chart has maximal mixed rank.  The fixed-order
transversality calculation \eqref{eq:torsion-transverse-open}
therefore makes every counted intersection simple.  Finally,
Proposition~\ref{prop:primary-density} supplies a full-rank torsion
point under the last hypothesis.  Its norms are one, so
Proposition~\ref{prop:unit-density} and the inverse function theorem
supply the required chart.
\end{proof}

The integer $s=2g+r$ is the rank of the semiabelian period lattice,
whereas the full Betti target has real dimension $2g+2r$.  These are
local counts on the unit-norm locus, not density statements in a
complex open set when $r>0$.  The eventual exact-order assertion is
likewise local and chart-dependent; it gives neither a threshold
uniform over families nor a global intersection-degree formula.

The critical-dimensional count also applies to transverse algebraic
slices in a base of larger dimension.  Using a torsion point to
choose the slice keeps the transversality condition algebraic.

\begin{corollary}[Transverse algebraic slices]\label{cor:transverse-slices}
Let $F_0\subseteq\C$ be algebraically closed, and let $\G/B$ and
$h$ be a smooth marked semiabelian family defined over $F_0$, of
relative dimension $d=g+r>0$.  If the mixed rank is full somewhere,
then there are an $F_0$-rational full-rank torsion point $x$ and a
smooth irreducible locally closed $F_0$-subvariety $Y\subseteq B$
of dimension $d$ through $x$ on which $h$ has full mixed rank.
The local counts of Corollary~\ref{cor:local-count} apply to this
slice, with exponent $2g+r$.
\end{corollary}
\begin{proof}
Proposition~\ref{prop:primary-density} supplies such a point $x$
over $F_0$, of some order $n$.  The algebraic vertical differential
$\Delta_{n,x}:T_xB\to\Lie(\G_x)$ in
\eqref{eq:full-torsion-differential} is surjective.  Choose an
$F_0$-linear complement $W$ to its kernel.  In a smooth affine
neighborhood of $x$, choose regular functions vanishing at $x$
whose differentials cut out $W$.  The common zero locus is smooth
of dimension $d$ at $x$.  Its component through $x$, restricted to
a smooth dense open containing $x$, is the required $Y$.
The restriction of $\Delta_{n,x}$ to $T_xY=W$ is an isomorphism,
so the mixed rank on $Y$ is full at $x$.  Since $x$ is torsion,
its norms are one; Proposition~\ref{prop:unit-density} and
Corollary~\ref{cor:local-count} now apply.
\end{proof}

\section{Saturated quotients and central monodromy}\label{sec:quotients}
\subsection{Coinvariants and algebraic semiabelian quotients}
Throughout this section assume \eqref{eq:full-relative}.  For a connected
rational normal subgroup $N\lhd P$, define the rational subspace
\begin{equation}\label{eq:WN}
 W_N=(N-1)E=(\Lie N)E.
\end{equation}
The first expression means the smallest rational subspace for which $N$ acts
trivially on the quotient, or equivalently the span supplied by the algebraic
coaction; it is not restricted to a particular set of integral points of $N$.
Connectedness gives the second expression in characteristic zero.

The next calculation explains why saturation removes exactly the relative dimension measured by the normal-orbit rank.

\begin{lemma}\label{lem:graded-WN}
The subspace $W_N$ is a rational Hodge subvariation contained in $W_{-1}E$, and
\begin{equation}\label{eq:graded-WN}
 \Gr^W_{-1}W_N=V_{p,N},\qquad
 \Gr^W_{-2}W_N=U_N.
\end{equation}
\end{lemma}
\begin{proof}
Normality makes $W_N$ stable under $P$, hence under every Hodge cocharacter of
the variation.  Connected monodromy is contained in $P$; after the fixed finite
cover, the subspace therefore gives a sub-local system and a Hodge subvariation.
It is contained in $W_{-1}E$, since the ambient group acts trivially on
$\Gr^W_0E$.

A rational weight cocharacter of a Levi of $P$ preserves $\Lie N$, so the
calculation can be made degree by degree in \eqref{eq:matrix-group}.
In degree $-1$, the images are the $p$-projection of $\Lie N$ and the images
of $V$ under its degree-zero part.  The latter already belong to $V_{p,N}$:
if $a\in\Gr^W_0\Lie N$ and $v\in V_p\subseteq\Lie P$, then
$[a,v]\in\Gr^W_{-1}\Lie N$ and its $p$-component is $a(v)$.

In degree $-2$, the possible images are $U_N$, $q_N(V)$, and the degree-zero
action on $\TT$.  For $a\in\Gr^W_{-1}\Lie N$ and $v\in V_p$, the bracket
$[a,v]=q_a(v)$ lies in $U_N$.  For $a\in\Gr^W_0\Lie N$ and $u\in U$, the
bracket $[a,u]$ likewise lies in $U_N$.  Thus no additional degree-$-2$ image
occurs.  Both reverse inclusions are immediate from the action on $e_0$.
\end{proof}

To turn a rational Hodge subvariation into an algebraic subgroup, we
must control the family and its ground field, not only the individual
complex fibres.  We first treat the abelian part through the relative
Hilbert scheme and rigidity; a common denominator will then handle
the residual extension.

\begin{lemma}\label{lem:relative-hom}
Let $F_0\subseteq\C$ be algebraically closed, let $B/F_0$ be smooth
and irreducible, and let $A_1,A_2$ be polarized abelian schemes over $B$.
An integral morphism
\[
 \eta:H_1(A_{1,\C}/B_\C,\Z)\longrightarrow
          H_1(A_{2,\C}/B_\C,\Z)
\]
of local systems that preserves the Hodge filtrations is induced, after
shrinking $B$ and a finite \'etale base change over $F_0$, by a
homomorphism of abelian schemes over $F_0$.  It agrees with $\eta$ on
any connected complex component containing a chosen matching fibre.
\end{lemma}
\begin{proof}
Fibrewise full faithfulness gives a unique homomorphism inducing
$\eta$.  These homomorphisms form a holomorphic family: locally a flat
integral matrix descends on the analytic quotients precisely because
it preserves the Hodge filtration.  The graphs have one locally
constant Hilbert polynomial for the product polarization, so one
polynomial suffices on the connected base.  Indeed the integral flat
map $\eta$ and the polarization classes determine the cohomological
class of its graph locally constantly.  Equivalently, if the source
has dimension $g_1$ and the polarization bundles are $L_1,L_2$, the
Hilbert polynomial of a fibre graph is
\[
 \frac{(c_1(L_1)+\eta^*c_1(L_2))^{g_1}}{g_1!}\,n^{g_1}.
\]
This follows from Riemann--Roch on an abelian variety, and its
coefficient is constant under flat transport.  The boundedness does
not concern an arbitrary unbounded collection of homomorphisms.

The relative morphism functor is an algebraic space locally of finite
presentation, and its graph map is an open immersion into the relative
Hilbert functor \cite[Tags~0D1B and~0D1C]{Stacks}.
For the fixed polynomial the projective Hilbert scheme is of finite
type.  Impose the closed conditions that the morphism preserve zero
and addition.  The resulting bounded Hom locus $\mathcal H\to B$
is separated and of finite type over $F_0$.
It is unramified: at a geometric homomorphism $a:A_{1,s}\to A_{2,s}$,
a first-order deformation fixing zero is a section of
$a^*T_{A_{2,s}}$, vanishing at zero.  This tangent bundle is trivial,
and $H^0(A_{1,s},\OO)=\overline{k(s)}$, so that section is zero.
The same infinitesimal calculation gives formal unramifiedness;
local finite presentation then gives unramifiedness.
Thus the fibres are zero-dimensional.

The holomorphic family lies generically in a component of
$\mathcal H_\C$ dominating $B_\C$.  Since $F_0$ is algebraically
closed, this is the base change of a component defined over $F_0$.
Its normalization is generically finite and separable over $B$.
After removing a proper closed subset of $B$, it is finite \'etale
and carries the universal homomorphism.  The induced homology map
agrees with $\eta$ at a fibre chosen on the holomorphic family.
Both maps are flat; hence they agree on the entire connected complex
component under consideration.
\end{proof}

We now treat the toric and abelian parts together.  A single denominator
in the flat Hodge splitting controls the residual extension on the
whole connected base.

\begin{lemma}\label{lem:algebraize}
Let $F_0\subseteq\C$ be algebraically closed, let $B/F_0$ be smooth and
irreducible, and let $\G_B/B$ be a semiabelian scheme.  Suppose that
$W\subseteq H_1(\G_{B,\C}/B_\C,\Q)$ is a rational Hodge subvariation
with Tate part $T_W$ and polarizable weight-$-1$ part $V_W$.
After a finite base change, shrinking, and a finite isogeny of the
ambient family, all defined over $F_0$, there is a connected semiabelian
subgroup $\HH/B$ whose rational realization is $W$.  Its relative
dimension is
\[
  \frac12\dim_\Q V_W+\dim_\Q T_W.
\]
The quotient $\G_B/\HH$ and the induced quotient section, when a marked
section is given, are defined over $F_0$.
\end{lemma}
\begin{proof}
First split the torus after a finite \'etale cover over $F_0$.
The saturated lattice in $T_W$ then defines a subtorus $T_1$ over
$F_0$: it is specified by an integral sublattice of a constant
cocharacter lattice.  We may choose a complementary torus after a
toric isogeny, also over $F_0$.

Choose a relative polarization of $\A$ \cite{Mumford1970,FaltingsChai1990}.  The orthogonal projector onto
$V_W$ is a rational flat Hodge idempotent.  One integer clears its
denominators in the integral local systems on the connected base.
On every complex fibre the resulting integral Hodge map on $H_1$ is
induced by a unique homomorphism of abelian varieties.  This is the
degree-one full faithfulness of the Hodge realization, not a
higher-codimension Hodge conjecture
\cite[\S10.1]{Deligne1974}.

Lemma~\ref{lem:relative-hom} algebraizes this integral map after a
finite cover and shrinking over $F_0$.  Its connected image, after
shrinking to make the kernel and image smooth over the base, is an
abelian subscheme $A_1$ realizing $V_W$.  In particular the cover and
$A_1$ are defined over $F_0$, rather than merely over $\C$.

Pass now to $\G_B/T_1$.  The image of $W$ is pure of weight $-1$
and maps isomorphically to $H_1(A_1/B,\Q)$.  It gives a rational
flat Hodge splitting of the extension restricted to $A_1$.
Clear a single denominator $D$ in this splitting: its rational
matrix is flat, so one denominator works on the connected base.
Write $\sigma$ for the rational Hodge splitting.  Fibrewise,
$D\sigma$ is integral and is therefore induced, by the full faithfulness
of the integral $1$-motive realization, by a homomorphism from $A_1$
to the restricted extension lifting $[D]:A_1\to A_1$.
Its existence splits the pullback extension by $[D]$.  The
Barsotti--Weil identification consequently says that the restricted
extension parameters are killed by this same $D$.
They are algebraic sections over $F_0$.  Consequently
\begin{equation}\label{eq:uniform-extension-torsion}
  [D]q\big|_{A_1}=0
\end{equation}
holds over $F_0$, since it holds after the faithfully flat extension
$F_0\subseteq\C$.

Push out the residual torus by multiplication by $D$.  We continue
to write $\G_B$ for the isogenous family.  Its restriction to $A_1$
now splits algebraically over $F_0$, because its extension class
vanishes.  This splitting is unique: $\Hom_B(A_1,T)=0$ for a torus
$T$.  The inverse image of the split copy of $A_1$ is a connected
semiabelian subgroup $\HH$ containing $T_1$.
Its rational realization is the original $W$: modulo $T_W$, the two
splittings coincide by uniqueness, and both contain precisely the
prescribed Tate subspace.  Equivalently, the difference of two
rational Hodge splittings would be a morphism from a pure
weight-$-1$ structure to a pure weight-$-2$ structure and hence zero.

The residual extension parameters annihilate $A_1$, so they come
from $(\A/A_1)^\vee$.  They define the quotient extension over
$\A/A_1$, and give the identity
\begin{equation}\label{eq:quotient-pullback-identity}
  \G_B/T_1\simeq
  \A\mathop{\times}_{\A/A_1}(\G_B/\HH).
\end{equation}
Thus the descended extension is $\G_B/\HH$; its pullback is
$\G_B/T_1$.  Every lattice, Hom component, toric pushout and
extension-parameter identity used above is defined over $F_0$.
A marked section is sent to its image by a homomorphism over $F_0$,
which proves the last assertion.
\end{proof}

The next statement includes the period-map compatibility needed in the
rank-reduction argument.  The existence of a quotient of mixed Shimura
data alone would not supply the algebraic semiabelian homomorphism.

\begin{proposition}\label{prop:saturated}
Let $N\lhd P$ be connected and rational.  After the permitted
refinements, $W_N$ defines a semiabelian subgroup $\HH_N$ of relative
dimension $h_N$.  The quotient marked family has relative dimension
$d-h_N$, preserves dense torsion, and satisfies
\eqref{eq:full-relative}.  Its classifying map is constant on the
relevant $N$-orbits.  If the original data are defined over an
algebraically closed $F_0\subseteq\C$, the refinements, quotient,
classifying map and its image closure can all be chosen over $F_0$.
\end{proposition}
\begin{proof}
Take $F_0=\C$ if no smaller field has been specified, and apply
Lemmas~\ref{lem:graded-WN} and~\ref{lem:algebraize}.
The homomorphism $f:\G_B\to\G'_B=\G_B/\HH_N$, together with the
identity on the lattice $\Z$, induces a morphism of marked
$1$-motives.  Its rational realization is the quotient map
$E\to E'=E/W_N$.
Choose polarization and level data for the quotient after a finite
refinement over $F_0$.  The canonical Hodge realization is functorial,
so its period map is the one obtained by applying this representation
quotient to the original period map.

Write $P'$ for the image of $P$ on $E'$.  Since $N$ acts trivially
on $E'$, the map of mixed Shimura data factors through the quotient
by $N$, by \cite[Proposition~2.9]{Pink1990}.  At compatible levels
there is a commutative diagram
\[
\begin{CD}
  \widetilde B @>{\widetilde\Phi}>> \DD_P
     @>{\rho_N}>> \DD_P/N @>{\bar\rho}>> \DD_{P'}\\
  @VVV @VVV @VVV @VVV\\
  B @>{\Phi}>> S_P @>{q_N}>> S_{P/N}
     @>{j}>> S_{P'} .
\end{CD}
\]
Here the lower maps between mixed Shimura varieties are algebraic;
$\DD_P/N$ denotes the domain in the quotient datum, not an
unqualified topological quotient.  The classifying map of $(\G'_B,h')$
is $j\circ q_N\circ\Phi$, followed by the universal realization
for its polarization type.  Its analytic map is constant on each
connected $N$-orbit.  Finite ambiguities in levels and integral
lattices are removed by the common refinement.  The resulting
algebraic map is therefore constant on every irreducible algebraic
image in question.  The quotient family and its level structure
are defined over $F_0$, so its classifying map and image closure
are defined over $F_0$ as well.

Because $f$ is a homomorphism, it sends fibrewise torsion to
fibrewise torsion.  The image of a dense torsion set is dense in
the closure of its image.  Finally, the image of the full group
$K=\Hom(\Q e_0,W_{-1}E)$ is
\[
  K'=\Hom(\Q e_0,W_{-1}E/W_N).
\]
By the tensor characterization, $P'$ is the generic Mumford--Tate
group of the quotient variation and its image on the negative
weights is their generic Mumford--Tate group.  It contains $K'$,
so it is the full preimage of that group.  This proves preservation
of relative nonspecialness and completes the proof.
\end{proof}

\begin{remark}
The construction uses $W_N=(N-1)E$, not $N\cap K$.  The terms $q_N(V)$ in
\cref{lem:graded-WN} are part of the necessary saturation.  A quotient that
removes only visible marked-point translations can leave a period-coordinate
quotient with no corresponding semiabelian homomorphism.
The normal Shimura quotient $q_N:S_P\to S_{P/N}$ must also be
distinguished from the classifying map induced by $E/W_N$.
The latter factors through $q_N$ via $j$ in the diagram above,
and $j$ can forget additional information; it is not asserted
to be generically finite.
\end{remark}

\subsection{Rank reduction and the normal-quotient formula}
We use the following form of weak mixed Ax--Schanuel.  If an irreducible analytic
germ $Z$ in the period domain has algebraic image closure $Y$ and weakly special
closure $T$, then
\begin{equation}\label{eq:weak-AS}
 \dim Z^{\Zar}+\dim Y\ge\dim T+\dim Z.
\end{equation}
This is a consequence of \cite{GaoKlingler2024,Chiu2025}; it is also stated in
the notation of normal functions in \cite[Theorem~4.4]{GaoZhang2026}.
Weakly special closures are the appropriate connected normal-monodromy orbits.
Their images are algebraic by the definability and algebraicity results for
mixed period maps \cite[Corollary~6.7]{BBKT2024}; see also \cite{BBT2023}.

We now combine the quotient with weak Ax--Schanuel.  The very-general-point choice is made before the leaf is selected.

\begin{proposition}\label{prop:rank-reduction}
Suppose $B\subseteq\MM_{g,r}$ satisfies \eqref{eq:full-relative} and
has mixed Betti rank $\ell<m=\dim B$.  Then there is another
universal-section problem satisfying \eqref{eq:full-relative}, with dimensions
\begin{equation}\label{eq:rank-reduction}
 d'=d-h_N,\qquad m'\le\ell-h_N
\end{equation}
for a connected rational normal subgroup $N$.  If $B$ and the
marked family are defined over an algebraically closed
$F_0\subseteq\C$, the new problem is defined over $F_0$.
If the original marked section has dense torsion, so does the new one.
In particular, $\ell<d$ implies $m'<d'$.
\end{proposition}
\begin{proof}
Choose a nonempty analytic open on which the mixed Betti rank is
constant and equal to $\ell$.  Choose $b$ very generally \emph{inside
this open}.  Exclude the proper Hodge loci and the nongeneric
fibre-dimension loci of every
rational subgroup-quotient classifying map that can arise after a finite
refinement.  This is a countable collection: rational subspaces of a fixed
rational realization, integral levels, polarizations, and their finite-index
refinements form countable sets.  On a finite cover, a proper closed exceptional
locus has proper closed image in $B$; it may therefore be excluded as well.
A nonempty analytic open cannot be covered by countably many proper
closed algebraic subsets.  Thus this choice does not presuppose the
algebraicity of any pointwise rank-drop locus.

Let $Z$ be the germ of the Betti fibre through $b$, of dimension
$t=m-\ell>0$, and let $Y$ be its algebraic image closure.  Since $b$ is
Hodge-generic, $Y$ has generic Mumford--Tate group $P$.  Its weakly special
closure $T$ is the orbit of a connected rational normal subgroup $N\lhd P$,
namely its connected algebraic monodromy, in the generic datum.
Equations \eqref{eq:weak-AS} and \eqref{eq:lift-closure-bound} give
\[
 \dim Y\ge h_N+t.
\]
Apply \cref{prop:saturated}.  The quotient classifying map is constant on $Y$.
Consequently its fibre through $b$ has dimension at least $h_N+t$.  By the
choice of $b$, this is a lower bound for its generic fibre dimension.  The
closure of the quotient classifying image has dimension at most
\[
 m-(h_N+t)=\ell-h_N.
\]
Take a smooth dense open of that image and its tautological section.  The
quotient construction preserves full relative translations.  When
torsion is dense, the homomorphism and the classifying-image map also
preserve that density.  Its relative dimension is $d-h_N$.
The field assertion follows from Proposition~\ref{prop:saturated},
including the field of definition of the image closure and its
smooth open.
\end{proof}

The same argument identifies the precise obstruction to full rank.
The quotient in the next statement is the normal mixed Shimura
quotient, not the possibly further classifying quotient of
Proposition~\ref{prop:saturated}.  We pass to a compatible finite
level cover so that the marked classifying map lifts to the mixed
Shimura variety $S_P$ of its generic datum, and retain the notation
$B$ for that lift.

\begin{corollary}[Normal-quotient rank formula]
\label{cor:normal-quotient-rank-formula}
In the universal-section setting with $K\subseteq P$, let
$q_N:S_P\to S_{P/N}$ be the normal mixed Shimura quotient for a
connected rational normal subgroup $N\lhd P$, at compatible levels.
Put
\[
  b_N=\dim\ol{q_N(B)}^{\Zar},\qquad
  h_N=\tfrac12\dim_\Q V_{p,N}+\dim(N\cap U).
\]
Then
\begin{equation}\label{eq:normal-quotient-rank-formula}
  \ell(B)=\min_N\bigl(b_N+h_N\bigr).
\end{equation}
These dimensions and the formula are unchanged by further finite
level refinements.
\end{corollary}
\begin{proof}
Fix $N$ and choose a point where $q_N|_B$ has its generic
differential rank $b_N$ and the mixed Betti rank is maximal.
Such a point exists because the first condition holds on a dense
Zariski open and the second on a nonempty analytic open.
The kernel of $d(q_N|_B)$ is contained in the tangent space to the
local $N$-orbit.  Lemma~\ref{lem:orbit-rank} bounds the complex
mixed rank on this kernel by $h_N$.  A complementary tangent
subspace has dimension $b_N$, so linear algebra gives
\[
  \ell(B)\le b_N+h_N.
\]

If $\ell(B)=m:=\dim B$, equality is attained for $N=1$.
Otherwise repeat the very-general-point choice in
Proposition~\ref{prop:rank-reduction}, also excluding the
fibre-dimension jumping loci for all rational normal Shimura
quotient maps.  These maps form a countable collection, including
their finite level refinements.  The Betti-leaf germ has dimension
$t=m-\ell(B)>0$.  Let $Y$ be its algebraic image closure and let
$N$ be the connected monodromy group of its weakly special closure.
The chosen point is Hodge-generic, so $N\lhd P$.
Weak mixed Ax--Schanuel and Corollary~\ref{cor:leaf-closure} give
$\dim Y\ge h_N+t$.  The normal quotient $q_N$ is constant on $Y$.
By the choice of the point, its fibre dimension there equals the
generic fibre dimension; hence
\[
  b_N\le m-(h_N+t)=\ell(B)-h_N.
\]
This is the reverse inequality for this $N$, proving the formula.
Finite changes of level preserve image dimensions and local
period-orbit ranks, which proves the last assertion.
\end{proof}

For weight-$-1$ normal functions, a normal-quotient rank formula is
given in \cite[Theorem~1.2]{GaoZhang2026}; compare also
\cite{Gao2020,GaoCorrigendum2021} in the abelian-scheme setting.
Here the term $\dim(N\cap U)$ records the additional weight-$-2$
directions, and the formula is a consequence of the orbit and
quotient calculations above.

\subsection{Abelian rank and the central directions}\label{sec:full-center}
Let $E_p$ be the variation of $[\Z\xrightarrow p\A_B]$:
\begin{equation}\label{eq:Ep}
 0\longrightarrow V\longrightarrow E_p\longrightarrow\Q(0)
 \longrightarrow0.
\end{equation}
Let $M_p$ be its connected algebraic monodromy group.

Maximal abelian Betti rank rules out a missing translation direction, even when the family has a fixed abelian part.

\begin{lemma}\label{lem:abelian-monodromy}
If $\rank_\R db_p=2g$, then $\Ru(M_p)=V$.
\end{lemma}
\begin{proof}
The pure monodromy of $V$ is reductive, because $V$ is polarizable.  The kernel
of $M_p$ over that pure monodromy is a vector subgroup of $V$; it is exactly
$V_0:=\Ru(M_p)$.  Andr\'e's normality theorem makes $V_0$ stable under the
generic Mumford--Tate group and hence a Hodge subvariation
\cite[Theorem~1 and its proof]{Andre1992}.

Suppose $V_0\ne V$, and put $V'=V/V_0$ and $E'=E_p/V_0$.  Its monodromy is
reductive, so the rational local-system extension $E'$ splits.  Its monodromy
invariants $I=(E')^{\mathrm{mon}}$ therefore surject onto $\Q(0)$.
The fixed-part theorem for admissible geometric mixed variations says that
$I$ is a constant mixed Hodge subvariation.  Its intersection with $V'$ is
$I\cap V'=(V')^{\mathrm{mon}}$.  The natural map
\[
 I\oplus_{I\cap V'}V'\longrightarrow E'
\]
is a surjection of equal-rank local systems, hence an isomorphism of mixed
Hodge variations by strictness.  Thus the quotient normal function is the
pushout of one constant extension on the fixed part.  Its Betti coordinates
are locally constant.  It follows that
\[
 \rank_\R db_p\le\dim_\Q V_0<2g,
\]
a contradiction.
\end{proof}

\begin{remark}
A split local-system extension need not be torsion when a fixed abelian factor
is present.  The preceding proof uses the stronger fixed-part assertion for
mixed Hodge variations and concludes \emph{constant Betti coordinates}, which
is the conclusion needed for the rank contradiction.
\end{remark}

The final central direction is detected by an elementary fact about algebraic units; algebraicity is essential here.

\begin{lemma}\label{lem:unit}
Let $B$ be a smooth connected complex quasi-projective variety and
$c\in\Gamma(B,\OO_B^\times)$.  If the connected algebraic monodromy of
$[\Z\xrightarrow c\Gm]$ is zero, then $c$ is constant.
\end{lemma}
\begin{proof}
The monodromy of a logarithm is a subgroup of $2\pi i\Z$.  Any nonzero subgroup
has Zariski closure $\Ga$, so the hypothesis makes all winding numbers zero.
Choose a smooth projective compactification after resolution.  If $c$ is
nonconstant, its principal divisor has a nonzero coefficient on a boundary
divisor; otherwise normality would extend $c$ to a unit on the projective
compactification.  At a general smooth point of that divisor, a transverse
punctured disc is contained in $B$.  Its winding number is the nonzero divisor
coefficient, a contradiction.
\end{proof}

The bracket with the full relative translation group propagates abelian
monodromy toward the center.  The only possible missing central
direction gives a constant multiplicative quotient.  Isolating that
quotient allows both torsion and small points to be used later.

\begin{lemma}\label{lem:constant-character}
Let a smooth marked semiabelian family over $\C$ satisfy
$K\subseteq P$ and $\rank_\R db_p=2g$, and let $M$ be its connected
algebraic monodromy.  If $M\cap U\ne U$, then, after a fixed finite
refinement and a toric pushout, there is a homomorphism from the marked
family to $\Gm$ whose restriction to the torus is nonzero and whose
value on the marked section is a constant $c\in\C^*$.  This constant
is not a root of unity.  If the original data are over $\Qbar$, all
these operations are over $\Qbar$ and $c\in\Qbar^*$.
\end{lemma}
\begin{proof}
Put $C_0=M\cap U$ and choose a nonzero integral character
$\chi:\TT\to\Q(1)$ annihilating $C_0$.
Lemma~\ref{lem:abelian-monodromy} makes the $p$-projection of
$\Gr^W_{-1}\Lie M$ equal to $V$.  Andr\'e's theorem gives
$M\lhd P$; hence its Lie algebra is stable under a rational weight
splitting of $P$.  For $a\in\Gr^W_{-1}\Lie M$ and
$v\in V_p\subseteq\Lie P$, normality gives
\[
 \chi q_a(v)=\chi[a,v]=0.
\]
For $b\in\Gr^W_{-1}\Lie P$ the same argument gives
\[
 0=\chi[a,b]=\chi q_a(p_b)-\chi q_b(p_a)=-\chi q_b(p_a).
\]
As the $p_a$ span $V$, the $\chi q$-projection of
$\Gr^W_{-1}\Lie P$ is zero.

The $1$-motive of $q_\chi=\sum_i\chi_iq_i$ therefore has no
Mumford--Tate unipotent extension.  Its rational mixed Hodge extension
splits, so $q_\chi$ is torsion in the isogeny category of $1$-motives
\cite[\S10.1]{Deligne1974}.  A single denominator in the generic
rational splitting gives a torsion section of fixed order; its
vanishing extends as an identity of algebraic sections.  Push out by
an integer killing it.  The corresponding rank-one extension splits
algebraically and uniquely, since $\Hom(\A,\Gm)=0$.  Projection to
its multiplicative factor is the required nonzero relative character;
write its marked value as $c$, initially an algebraic unit on the base.

Its connected logarithmic monodromy is zero.  Indeed, normality makes
$\Lie M$ a $P$-stable mixed Hodge subobject of $\End(E)$.  The map
induced by the multiplicative quotient to its logarithmic Lie algebra
is a Hodge morphism with pure weight-$-2$ target.  Strictness says that
its image comes from
\[
 W_{-2}\Lie M=\Lie(M\cap U)=\Lie C_0,
\]
which the character annihilates.  Lemma~\ref{lem:unit} now makes $c$
constant.  Were $c$ a root of unity, the marked $1$-motive in this
nonzero toric quotient would split rationally.  Its relative
Mumford--Tate translation would then be zero, contrary to the image
of the full group $K$.

For data over $\Qbar$, the character and pushout use integral matrices,
the extension-parameter identity is algebraic, and the splitting is
unique.  They are therefore defined over $\Qbar$ after the indicated
fixed refinements.  The resulting constant unit belongs to $\Qbar^*$.
\end{proof}

A single torsion value excludes the constant quotient in the lemma.
No density assumption is needed once the abelian rank is known.

\begin{proposition}\label{prop:full-center}
Let the smooth marked family over $\C$ satisfy $K\subseteq P$ and
$\rank_\R db_p=2g$.  If $h$ has at least one torsion specialization,
then
\begin{equation}\label{eq:full-center}
 M\cap U=U.
\end{equation}
\end{proposition}
\begin{proof}
A torsion value is carried by every fixed homomorphism to a torsion
value.  Thus the constant supplied by
Lemma~\ref{lem:constant-character} would be a root of unity, a
contradiction.  Here the fixed torus-splitting cover is finite \'etale,
the identity killing $q_\chi$ extends over the whole smooth base, and
the split extension is unique.  Thus no shrinking discarding the
specified torsion specialization is used.
\end{proof}

For a relatively nonspecial torsion-dense section,
\eqref{eq:abelian-rank} supplies the abelian-rank hypothesis.
The small-points replacement for the final root-of-unity argument is
proved in Lemma~\ref{lem:small-center}.

\section{Canonical heights and arithmetic support}\label{sec:adelic}
\subsection{Adelic maxima on dominating models}
In Sections~\ref{sec:adelic}--\ref{sec:support} and the arithmetic
application in Section~\ref{sec:central-bound}, the base is defined over
a number field $F$.  All geometric components are taken after a finite extension
of $F$ when necessary.  An overline denotes an adelic metrized line bundle or
adelic Cartier divisor in the boundary completion of Yuan--Zhang
\cite{YuanZhang2026}.  The underlying geometric adelic boundary class is denoted
without an overline.  These underlying classes retain data on projective
compactifications; they are not classes only in $\Pic(B)$.

Heights use the normalized local degrees, and arithmetic intersection
numbers over $F$ are divided by $[F:\Q]$.  Geometric intersection
numbers are degrees over $F$ and have no such factor.  These conventions
make both heights and normalized arithmetic intersections unchanged by
finite extension of the ground field.  For a nef adelic line bundle
$\bL$ with big geometric part on an $m$-fold, the fundamental inequality
of \cite[Theorem~5.3.3]{YuanZhang2026} reads
\begin{equation}\label{eq:fundamental}
 0\le \bL^{m+1}\le (m+1)e_1(\bL)L^m,
\end{equation}
where $e_1$ is the essential minimum for the normalized height.
We also use multilinearity, boundary-topology continuity, and positivity
of the intersection of nef bundles with an integrable effective divisor
\cite[Proposition~4.1.1 and its proof]{YuanZhang2026}.
All theorem numbers for this source refer to the author version specified
in the bibliography.

A strongly nef bundle is a boundary-topology limit of nef model bundles.
A nef bundle is one admitting a strongly nef auxiliary bundle $\overline M$
such that $a\bL+\overline M$ is strongly nef for every positive integer $a$.
In particular every nef bundle is integrable.  Distinguishing these notions is
important in the maximum construction.

Identified interior restrictions allow maxima to be formed on dominating
models.  The construction must retain exceptional valuations.  For
related lattice operations on adelic divisors see
\cite[Lemma~2.2.3]{Hultberg2026}.  We include the model argument with
its fixed interior identifications in the Yuan--Zhang category used
here.

\begin{lemma}\label{lem:max-nef}
Let $\overline L_1,\ldots,\overline L_s$ be nef adelic line bundles whose
underlying algebraic restrictions to $B$ have specified identifications with
one line bundle.  Use those identifications to regard their model divisors and
local metric potentials in a common rational trivialization.  Their maximum,
taken on dominating normal models, is a nef adelic line bundle with that same
interior restriction.
\end{lemma}
\begin{proof}
We first fix the common interior on which all identifications live.
For a finite collection of model bundles, choose a normal arithmetic
model $\mathcal U$ of $B$ over a localization of $\OO_F$ and a line
bundle $\mathcal L$ on $\mathcal U$ realizing their identified
restrictions.  Clearing denominators and localizing at finitely
many further primes extends the given identifications to
$\mathcal U$.  Choose a normal projective arithmetic compactification
and dominate all the finitely many models.  After one integral
multiple, represent the model bundles by Cartier divisors $D_j$
whose rational identifications agree on $\mathcal U$.
The geometric construction is identical if the ground field has
no arithmetic places.  The sum of the fractional ideals
\[
 I=\OO(D_1)+\cdots+\OO(D_s)
\]
is a coherent rank-one fractional ideal.  On the normalized blowup
principalizing $I$, its invertible image is a quotient of
$\bigoplus_j\OO(D_j)$.  The quotient is geometrically nef: on each complete curve in a geometric
fibre, some summand maps nontrivially to it, so its degree is at least
the nonnegative degree of that summand.  The same argument applies to
projective models over a field.  Arithmetic nefness will follow from
the quotient metrics constructed below.

At a divisorial valuation $v$, the valuation of $I$ is the minimum of the
valuations of its summands.  The associated Cartier divisor therefore has
coefficient $\max_jv(D_j)$.  On any further model its pullback has the same
valuation rule.  Thus the construction is the Cartier b-maximum.  It is not
the operation of taking coefficients of the traces on one model and ignoring
new exceptional divisors.

At an archimedean place write $g_j$ for the metric potentials in the common
frame.  Equip the quotient of the direct sum of the $n$th powers with its
Hermitian quotient metric.  The normalized potential is
\[
 g^{(n)}=\frac{1}{2n}\log\left(\sum_{j=1}^s e^{2ng_j}\right),
 \qquad 0\le g^{(n)}-\max_jg_j\le\frac{\log s}{2n}.
\]
It is semipositive: locally it is a log-sum-exp of plurisubharmonic potentials,
after including the holomorphic coefficients of the quotient map.
The map from each summand is a contraction.  Along the normalization of an
arithmetic horizontal curve, some summand maps nontrivially to the quotient;
the resulting effective small divisor shows that the arithmetic degree of
the quotient is at least that of the summand.  At nonarchimedean places use
the quotient sup metric.  To compare the powers, use the common
rational frame and write $v(I)$ for the minimum valuation of local
generators of a fractional ideal.  For every divisorial valuation,
\[
 v\!\left(\sum_j\OO(nD_j)\right)
   =n\min_j v(\OO(D_j))
   =n\,v\!\left(\sum_j\OO(D_j)\right).
\]
Thus passing to normal blowups for the powers, when needed, does
not change their normalized Cartier b-divisor.  Uniform convergence
gives the maximum metric and nefness.

For strongly nef bundles, choose their defining model sequences.
Each sequence has a fixed algebraic restriction on an arithmetic
interior, by the definition of a model Cauchy sequence in
\cite[\S2.5]{YuanZhang2026}.  There are only finitely many sequences:
intersect their interiors and extend the finitely many prescribed
identifications there as above.  Thus one \emph{fixed} common
$\mathcal U$ and $\mathcal L$ work for every term.  Use a diagonal
system of normal models dominating all terms under consideration;
the identifications are always pulled back from $\mathcal U$.
Perform the model construction term by term.  Choose the smoothing
power tending to infinity so that the log-sum-exp error tends to
zero in the boundary norm as well as uniformly on compact sets.
This is possible after increasing the fixed positive boundary
order unit by a constant archimedean metric.  For any boundary
order unit $\overline D_0$,
\[
 |\max_j A_j-\max_j B_j|\le\max_j|A_j-B_j|
\]
pointwise on potentials and on every divisorial valuation.  More
explicitly, $-\delta\overline D_0\le A_j-B_j\le
\delta\overline D_0$ for all $j$ implies the same two inequalities
for the difference of their maxima.  Consequently the maximum
operation is Lipschitz for the boundary topology.  The approximations converge
to the asserted strongly nef maximum, independently of the models.

Finally suppose only that the bundles are nef.  Choose strongly nef
auxiliaries $\overline M_j$ in their definitions and put
$\overline M=\sum_j\overline M_j$.  Each $a\overline L_j+\overline M$ is
strongly nef.  The strongly nef case and
\[
 \max_j(a\overline L_j+\overline M)
 =a\max_j\overline L_j+\overline M
\]
prove the assertion for nef bundles.
\end{proof}

\subsection{Quadratic compensation and canonical heights}\label{sec:canonical-height}
Choose a symmetric relatively ample rigidified theta bundle $\Theta$ and its
canonical nef adelic metric.  Write $q_i=\lambda a_i$, using its principal
polarization.  Set
\begin{equation}\label{eq:LQ}
 \bL=p^*\bTheta,\qquad \bQ=\sum_{i=1}^r a_i^*\bTheta.
\end{equation}
We normalize the Poincar\'e line bundle so that its pullback by
$(x,\lambda y)$ is the rigidified cross term of the quadratic theta identity.
For the $i$th toric pushout $h_i$, let
\[
 \bD_i=\widehat{\divi}(h_i).
\]
Its underlying divisor on $B$ is zero, but its geometric adelic boundary class
need not be zero.  In the adelic Picard group,
\begin{equation}\label{eq:theta-cross}
 \bD_i=(p+a_i)^*\bTheta-p^*\bTheta-a_i^*\bTheta.
\end{equation}
The identity is with the canonical metric and rigidifications, not merely
numerical equivalence.  It follows from the cubical theta identity and
uniqueness of the canonical metric.  The symmetric and antisymmetric canonical
extensions used here are supplied by
\cite[Theorems~6.1.1 and~6.1.3]{YuanZhang2026}.

For $\chi\in\Z^r$ put
\begin{equation}\label{eq:DTchi}
 \bD_\chi=\sum_i\chi_i\bD_i,
 \qquad \bT_\chi=\max(\bD_\chi,-\bD_\chi).
\end{equation}
The maximum is initially a boundary-completion divisor with Green function
$|u_\chi|$, where
\[
 u_i=-\log\|h_i\|,\qquad u_\chi=\sum_i\chi_i u_i.
\]
It is effective.  Its integrability follows from the next lemma; it must not be
assumed just because its interior restriction is trivial.

Although the absolute-value divisors need not themselves be nef, a quadratic theta correction controls all signs simultaneously.

\begin{lemma}\label{lem:compensation}
For a finite set $\Xi\subseteq\Z^r$ and positive rational weights $w_\chi$, put
\[
 \bT_w=\sum_{\chi\in\Xi}w_\chi\bT_\chi.
\]
There is a rational $R>0$ such that, for every rational $\eps>0$,
\begin{equation}\label{eq:N-eps}
 \bN_{\eps,w,R}=\bL+\eps\bT_w+R\eps^2\bQ
\end{equation}
is nef.  The same $R$ works for the weights in a fixed sufficiently small
bounded neighborhood.  Every $\bT_\chi$ is integrable.
\end{lemma}
\begin{proof}
For signs $\sigma_\chi\in\{\pm1\}$, let
$v_i=\sum_\chi w_\chi\sigma_\chi\chi_i$.  Quadraticity gives
\begin{align}\label{eq:quadratic-completion}
 \bL+\eps\sum_i v_i\bD_i+R\eps^2\bQ
 ={}&\bTheta\left(p+\eps\sum_i v_i a_i\right)\notag\\
 &+\eps^2\left(
 R\sum_i\bTheta(a_i)-\bTheta\left(\sum_i v_i a_i\right)\right).
\end{align}
The notation with rational arguments is shorthand for denominator-cleared
quadratic identities.  For example,
$\bTheta(p+\eps a)=n^{-2}\bTheta(np+n\eps a)$ when $n\eps$ is integral.
No rational division section has been postulated.

The second line of \eqref{eq:quadratic-completion} corresponds to the rational
matrix $RI-vv^{\mathsf t}$.  For $R\ge\|v\|^2$ it is positive semidefinite.
Rational completion of squares expresses it as a sum of rank-one rational
matrices with nonnegative rational coefficients.  Hence it is a nonnegative
sum of canonical theta pullbacks and is nef.  One explicit rational choice is
\begin{equation}\label{eq:explicit-compensation}
  R=1+\sum_{i=1}^r\left(\sum_{\chi\in\Xi}
                  w_\chi|\chi_i|\right)^2.
\end{equation}
Indeed $|v_i|\le\sum_\chi w_\chi|\chi_i|$, for every sign choice.
On a fixed bounded weight box, replace each $w_\chi$ in
\eqref{eq:explicit-compensation} by a rational upper bound.
This gives a single rational $R$ for that entire box.

The interior line bundles on the left of
\eqref{eq:quadratic-completion} have the same specified restriction, because
the $h_i$ trivialize the Poincar\'e factors on $B$.  Taking their b-maximum and
using \cref{lem:max-nef} gives exactly \eqref{eq:N-eps}.  With a single
character and one fixed rational $\eps>0$, it expresses $\bT_\chi$ as a
rational linear combination of nef bundles, proving integrability.
\end{proof}

The same quadratic calculation applies to support functions of rational
polytopes, not only to sums of absolute values.  This extension makes
clear which part of the positivity is independent of the chosen
character basis.  Polyhedral canonical metrics in the semiabelian
setting also occur in \cite[\S3.1]{Hultberg2026}.

\begin{proposition}\label{prop:polyhedral}
Let $\Pi\subset\R^r$ be a rational polytope containing zero and put
\[
 \bT_\Pi=\max_{v\in\operatorname{Vert}(\Pi)}\sum_i v_i\bD_i.
\]
This divisor is effective and integrable.  For any positive rational
$\eps$, the class $\bL+\eps\bT_\Pi+R\eps^2\bQ$ is nef whenever
$R\in\Q_{>0}$ satisfies
$R\ge\max_{v\in\operatorname{Vert}(\Pi)}\|v\|^2$.
\end{proposition}
\begin{proof}
For each vertex, \eqref{eq:quadratic-completion} applies, with
$RI-vv^{\mathsf t}$ positive semidefinite.  The resulting nef classes
have the same specified restriction to the interior.  Their maximum is
therefore nef by Lemma~\ref{lem:max-nef}.  Since $0\in\Pi$, the
support function of $\Pi$ is nonnegative; this proves effectivity
both for metric potentials and for all divisorial valuations.  Fixing
one $\eps>0$ expresses $\bT_\Pi$ as a rational linear combination of
nef classes and proves integrability.
\end{proof}

We now make \eqref{eq:intro-height} precise.  Fix the splitting of the
torus and the rigidifications of the Poincar\'e bundles.  Set
$\bT_0=\sum_{i=1}^r\bT_{e_i}$.  For an algebraic point $b$ choose a
number field $F_b$ over which the point and all data are defined, and
use the local weights $n_v=[(F_b)_v:\Q_v]/[F_b:\Q]$, with the usual
archimedean convention.  Write
\begin{equation}\label{eq:canonical-height}
 \begin{gathered}
 \hab(h(b))=h_{\bL}(b),\qquad \htot=\hab+\htor,\\
 \htor(h(b))=h_{\bT_0}(b)=\sum_vn_v\sum_i|u_{i,v}(b)|.
 \end{gathered}
\end{equation}
These expressions are independent of enlarging $F_b$.  The first is
the fibrewise N\'eron--Tate height; the second uses the canonically
metrized, rigidified toric pushouts.  No arbitrary constant metric is
included in their definition.  Both are nonnegative.  For a split
torus the latter is twice the sum of the usual coordinate Weil heights.
For the abelian-plus-boundary decomposition in the fixed semiabelian
setting, see \cite{ChambertLoir1999}; for the canonical metrics and
functoriality, see \cite[\S\S2.3--2.5]{KuehneSmall2022}.

The compatibility of the Poincar\'e metric with multiplication gives,
for every positive integer $n$,
\begin{equation}\label{eq:height-scaling}
 \hab([n]x)=n^2\hab(x),\qquad
 u_{i,v}([n]x)=nu_{i,v}(x),\qquad
 \htor([n]x)=n\htor(x).
\end{equation}
Indeed the underlying rigidified isomorphism for the partial
multiplication $[n]$ on the abelian-point coordinate is an isometry:
this follows from uniqueness of the canonical metrics in
\cite[Theorems~6.1.2--6.1.3]{YuanZhang2026}.  The group law on the
extension uses this very isomorphism.  In particular, at a torsion
specialization all local norms are one and
\begin{equation}\label{eq:torsion-height}
 h_{\bL}=h_{\bT_\chi}=0,\qquad
 h_{\bN_{\eps,w,R}}=R\eps^2h_{\bQ}.
\end{equation}
Conversely, if $\htot(x)=0$, then $p(x)$ is torsion by positive
definiteness of the N\'eron--Tate height over algebraic numbers.
For its order $n$, the point $[n]x$ belongs to the torus.  Its
coordinate Weil heights vanish by \eqref{eq:height-scaling} and the
product formula.  Kronecker's theorem makes its coordinates roots of
unity; thus $x$ is torsion as well.

For small points it is essential to compare heights without an additive
error.  The following family version of canonical-height functoriality
will also be used for the saturated quotients; compare
\cite[Lemmas~9--10]{KuehneSmall2022}.

\begin{lemma}\label{lem:height-functorial}
Let $\varphi:\G\to\G'$ be a fixed homomorphism of semiabelian schemes
over a common smooth base over $\Qbar$, with abelian part
$f:\A\to\A'$.  Choose toric character bases and symmetric relatively
ample polarizations.  After shrinking there is $C>0$ such that
\[
 \htor'(\varphi(x))\le C\htor(x),\qquad
 \hab'(f(p(x)))\le C\hab(x)
\]
for every algebraic point in the family.  A fixed isogeny preserves the
property of height tending to zero in both directions, including for
arbitrary lifts of points.  Finite base changes preserve these heights.
\end{lemma}
\begin{proof}
Write the map on toric characters as an integral matrix $(c_{ji})$.
Compatibility of extension classes gives
\[
 f^\vee(q'_j)=\sum_i c_{ji}q_i.
\]
The corresponding rigidified Poincar\'e isomorphism identifies the
$j$th target pushout with the indicated tensor product of source
pushouts.  It is an isometry by canonical-metric uniqueness, and the
homomorphism $\varphi$ respects this identification.  There is no
scalar ambiguity: the rigidification fixes the scalar at the identity,
and two group homomorphisms with the same abelian and toric parts
differ through $\Hom(\A,\mathcal T')=0$.  Consequently
\begin{equation}\label{eq:exact-norm-functoriality}
 u'_{j,v}(\varphi(x))=\sum_i c_{ji}u_{i,v}(x)
\end{equation}
at every place.  The triangle inequality proves the toric comparison.

Choose an integer $C$ such that $C\Theta-f^*\Theta'$ is ample on the
geometric generic fibre.  It is relatively ample after shrinking and
is symmetric.  Nonnegativity and linearity of its canonical height give
\[
 0\le\widehat h_{C\Theta-f^*\Theta'}(p(x))
   =C\hab(x)-\hab'(f(p(x))).
\]
This proves the abelian comparison, with no $O(1)$ term.
For an isogeny, choose a homomorphism in the reverse direction whose
composition is multiplication by a fixed positive integer.  Apply the
two comparisons and \eqref{eq:height-scaling}; this controls the height
of every lift.  Finally, normalized local degrees and canonical
rigidifications make the heights unchanged by finite base extension.
\end{proof}

Changing a torus basis or a polarization gives comparisons in both
directions.  On a nonsplit family, height statements are read after one
fixed finite splitting refinement.  Existence of generic small points,
and existence of a positive gap outside a proper closed subset, do not
depend on that choice: use the comparisons on a common finite cover,
and take the finite images of exceptional sets.  This convention is
used in all the arithmetic applications below.

\subsection{Mixed Chern measures and all-character support}\label{sec:support}
Assume $m=g+k$, with $k\ge0$, and that the abelian Betti map has rank $2g$.
We do not assume positivity of the leading intersection introduced below.
We work over a number field throughout this section.  The
zero-dimensional small-points case is handled by the height-zero
criterion in Section~\ref{sec:canonical-height}; below we take $m\ge1$.
At an archimedean place use $d^c=(\partial-\bar\partial)/(2\pi i)$ and put
\[
 \omega=p^*c_1(\bTheta).
\]
Thus $\omega$ is smooth semipositive and is the pullback of the polarized
Betti form.  Its complex rank is at most $g$ and is $g$ on the abelian
maximal-rank locus.

The canonical Poincar\'e metric is
\begin{equation}\label{eq:norm-coordinate}
 u_i=2\pi\left(\operatorname{Im}t_i
       -\operatorname{Im}w_i\,(\operatorname{Im}\tau)^{-1}
            \operatorname{Im}z\right).
\end{equation}
This is \cite[(2.15)]{BurgosHolmesDeJong2018} in the coordinates of
Section~\ref{sec:betti}.  Reversing the Poincar\'e convention changes all the
corresponding signs and leaves $|u_\chi|$ and the argument unchanged.
From \eqref{eq:theta-cross},
\[
 \kappa_i:=\ddc u_i=\omega_{p+a_i}-\omega_p-\omega_{a_i}.
\]
Locally the polarized Betti form is a constant alternating form in the real
Betti coordinates.  Expanding the last expression therefore leaves only cross
terms, each containing a differential of a Betti coordinate of $p$.
Consequently
\begin{equation}\label{eq:curvature-cancel}
 \omega^g\wedge\kappa_i=0,\qquad \omega^{g+1}=0.
\end{equation}

Curvature cancellation removes the indefinite terms and leaves positive mixed measures, even though the individual factors need not have positive curvature.

\begin{lemma}\label{lem:positive-measures}
For any integral characters $\chi_1,\ldots,\chi_k$, the measure
\begin{equation}\label{eq:mixed-measure}
 \mu_{\chi_1,\ldots,\chi_k}
 =\omega^g\wedge\ddc|u_{\chi_1}|\wedge\cdots\wedge\ddc|u_{\chi_k}|
\end{equation}
is well-defined, positive, and finite.  Moreover,
\begin{align}
 \int_{B(\C)}\mu_{\chi_1,\ldots,\chi_k}
   &=L^gT_{\chi_1}\cdots T_{\chi_k},\label{eq:total-mass}\\
 L^{g+1}E_1\cdots E_{m-g-1}&=0\label{eq:geometric-vanish}
\end{align}
where, in the second formula, $m\ge g+1$ and every $E_i$ is the
geometric image of a global integrable adelic bundle $\overline E_i$
over the chosen number field.
\end{lemma}
\begin{proof}
Choose smooth convex approximations $f_\delta$ of $|x|$, with
$|f_\delta'|\le1$ and $f_\delta''\ge0$.  Locally,
\[
 \ddc f_\delta(u_\chi)
 =f_\delta'(u_\chi)\ddc u_\chi
  +f_\delta''(u_\chi)\,du_\chi\wedge d^cu_\chi.
\]
After wedging with $\omega^g$, every term of the first kind vanishes by
\eqref{eq:curvature-cancel}.  The remaining terms are positive.
On a relatively compact coordinate chart, the negative parts of
$\ddc f_\delta(u_\chi)$ have a common smooth lower bound.  Adding one fixed
strictly plurisubharmonic smooth potential makes all these potentials locally
bounded plurisubharmonic functions, converging uniformly.  The mixed
Bedford--Taylor products therefore converge and are independent of the
regularization \cite{BedfordTaylor1982}.  This agrees with the multilinear
products of the integrable metrics in \cref{lem:compensation} and proves
positivity.

These metrics come from global integrable adelic bundles over the
number field $F$.  The mixed total-mass formula in
\cite[Lemma~5.4.4]{YuanZhang2026} therefore gives
\eqref{eq:total-mass}, including finiteness.  For
\eqref{eq:geometric-vanish}, choose the global integrable lifts
$\overline E_i$ in the statement and apply that same multilinear
formula with $g+1$ smooth factors $\omega$.  The local product
vanishes because $\omega^{g+1}=0$.  This scope includes all subsequent
factors: $L,Q,H,T_\chi$ are geometric images of the global integrable
bundles already constructed, as are their linear combinations.
Guo's theorem removes the global-origin
hypothesis and gives the corresponding local statement
\cite[Theorem~1.2]{Guo2025}; see also
\cite[Remark~5.4.5]{YuanZhang2026}.
There is no additional boundary mass to be appended: the geometric adelic
boundary intersection is exactly what the total-mass theorem computes.
\end{proof}

Let
\begin{equation}\label{eq:Delta}
 \Delta=L^gT_0^k\ge0,\qquad T_0=\sum_{i=1}^rT_{e_i}.
\end{equation}
For $k=0$ this is $L^g$, which is positive by the maximal abelian Betti rank.

Vanishing of the leading intersection is not a reason to discard a component.  It forces all the corresponding character measures to vanish.

\begin{lemma}\label{lem:zero-leading}
If $k\ge1$ and $\Delta=0$, then every measure
$\mu_{\chi_1,\ldots,\chi_k}$ in \eqref{eq:mixed-measure} is zero.
\end{lemma}
\begin{proof}
For any finite positive character sum $T_w$, the triangle inequality in each
Green function and at each divisorial valuation gives
\[
 0\le T_w\le A T_0
\]
for some positive rational $A$.
We first prove monotonicity of the leading degree on such sums.  If $S\ge T$
are two such sums, take a common compensation constant $R$ and write
\[
 N_S=L+\eps S+R\eps^2Q,\qquad
 N_T=L+\eps T+R\eps^2Q.
\]
Both are nef.  Effectivity of $S-T$ gives
\[
 L^g(N_S^k-N_T^k)
 =\eps\sum_{j=0}^{k-1}L^gN_S^jN_T^{k-1-j}(S-T)\ge0.
\]
Every term containing an additional factor of $L$ vanishes by
\eqref{eq:geometric-vanish}.  Dividing by $\eps^k$ and taking $\eps\to0^+$
therefore proves $L^gS^k\ge L^gT^k$.  It follows that
\[
 0\le L^gT_w^k\le A^k\Delta=0.
\]
Take $T_w$ containing all the characters under consideration with positive
weights and expand its $k$th power.  Each coefficient is the mass of a positive
measure by \cref{lem:positive-measures}, so each mass is zero.  A positive
measure of zero total mass is zero.
\end{proof}

Suppose $\Delta>0$, and choose $T_w$ containing the coordinate characters with
positive weights.  Then
\begin{equation}\label{eq:leading-degree}
 N_\eps^m=A_w\eps^k+O(\eps^{k+1}),\qquad
 A_w=\binom{m}{g}L^gT_w^k>0.
\end{equation}
Indeed, \eqref{eq:geometric-vanish} kills every term with more than $g$ factors
of $L$, and a factor of $Q$ costs two powers of $\eps$.

For a positive leading intersection, small total height already controls
the extension-parameter height.  The argument first fixes one
compensation parameter; uniform control of the small-section constant
as that parameter tends to zero is unnecessary.

\begin{lemma}\label{lem:bounded-height}
Work over a number field, with $m=g+k\ge1$ and maximal abelian Betti
rank.  Suppose $\Delta>0$ and $b_n\in B(\Qbar)$ is a generic sequence
such that $\htot(h(b_n))\to0$.  Then $h_{\bQ}(b_n)$ is bounded after
finitely many terms are discarded.
\end{lemma}
\begin{proof}
Choose a nef ample model height $\bH$ that majorizes the fixed
integrable height $\bQ$ off a proper closed set:
\begin{equation}\label{eq:height-majorization}
 h_{\bQ}\le C_Qh_{\bH}+C_0.
\end{equation}
This follows from boundary-order comparison and an ample model divisor
\cite[Lemma~5.1.7]{YuanZhang2026}.  Fix first a positive character sum
$T_w$ containing the coordinate characters.  By
\eqref{eq:geometric-vanish},
\[
 N_\eps^{m-1}H=
 \begin{cases}O(\eps^{k-1}),&k\ge1,\\O(1),&k=0.\end{cases}
\]
Together with \eqref{eq:leading-degree}, the geometric Siu inequality
makes $N_\eps-c\eps H$ big for one $c>0$ and all sufficiently small
positive rational $\eps$.  Adding a sufficiently large constant
metrization makes it arithmetically big.  A nonzero small section gives
\begin{equation}\label{eq:height-lower}
 h_{\bN_\eps}\ge c\eps h_{\bH}-C_\eps
\end{equation}
outside its proper zero locus.  The bigness and metric-shift statements
are \cite[Theorem~5.2.2 and Lemma~5.2.10]{YuanZhang2026}.

At the chosen points the exact identity replacing the torsion identity is
\begin{equation}\label{eq:small-height-identity}
 h_{\bN_\eps}(b_n)=\hab(h(b_n))+\eps h_{\bT_w}(b_n)
                           +R\eps^2h_{\bQ}(b_n),
 \qquad 0\le h_{\bT_w}\le C_w\htor.
\end{equation}
Fix one $\eps_0>0$ with $R\eps_0C_Q<c/2$.  Genericity permits
removal of the fixed exceptional sets.  Combining the preceding
inequalities gives
\begin{align*}
 (c\eps_0-R\eps_0^2C_Q)h_{\bH}(b_n)
 &\le C_{\eps_0}+\hab(h(b_n))\\
 &\quad+\eps_0C_w\htor(h(b_n))+R\eps_0^2C_0.
\end{align*}
The right side is bounded.  Thus $h_{\bH}(b_n)$, and then
$h_{\bQ}(b_n)$ by \eqref{eq:height-majorization}, are bounded.
\end{proof}

Fix the generic sequence just obtained, and let $C$ bound its
$h_{\bQ}$-values.  For any fixed finite character collection and
positive weights, \eqref{eq:small-height-identity} gives
\[
 e_1(\bN_{\eps,w,R})\le RC\eps^2,
\]
because the first two terms tend to zero along a generic sequence.
The same sequence works as the character weights vary.  Assume that
the coordinate characters occur with positive weights, and choose one
rational $R$ on a bounded open weight box.  Then
\eqref{eq:fundamental} and \eqref{eq:leading-degree} give
\begin{equation}\label{eq:arithmetic-small}
 0\le\bN_{\eps,w,R}^{m+1}=O(\eps^{k+2}).
\end{equation}
The left side is a polynomial in $\eps$ and the weights.  Its
coefficients of $\eps$-degrees at most $k+1$ vanish for every rational
weight vector in the box, hence vanish identically as polynomials.
Differentiating with respect to the weight of an included character
$\psi$, and dividing by $(m+1)\eps$, gives
\begin{equation}\label{eq:mixed-arithmetic-small}
 \bN_{\eps,w,R}^{m}\bT_\psi=O(\eps^{k+1}).
\end{equation}
Thus the differentiation is a polynomial identity, rather than a
differentiation of a pointwise estimate with uncontrolled constants.

The resulting support statement includes exact torsion and arbitrary
generic small sequences, as well as the case of zero leading
intersection.

\begin{proposition}[All-character support]\label{prop:all-character-support}
Let the marked family be defined over a number field, with
$m=g+k\ge1$ and $\rank_\R db_p=2g$.  Suppose that it has a generic
sequence of algebraic points with $\htot\to0$.  For every collection
of integral characters,
\begin{equation}\label{eq:all-character-support}
 \Supp\mu_{\chi_1,\ldots,\chi_k}\subseteq\{u_1=\cdots=u_r=0\}.
\end{equation}
No positivity of $\Delta$ is assumed.  In particular the conclusion
holds when fibrewise torsion is Zariski dense.
\end{proposition}
\begin{proof}
If $\Delta=0$, this is \cref{lem:zero-leading}.  Suppose $\Delta>0$.
At an archimedean place choose a nonnegative smooth compactly supported
function $\varphi\le|u_\psi|$, with conjugation symmetry when appropriate.
The metrized trivial bundle with Green function $\varphi$ is integrable;
adding a sufficiently large ample model metric dominates its curvature.
The divisor $\bT_\psi-\bO(\varphi)$ is effective, so positivity gives
\[
 \bN_\eps^m\bT_\psi\ge c_v\int\varphi\,c_1(\bN_\eps)^m,
 \qquad c_v>0.
\]
Divide by $\eps^k$ and use \eqref{eq:mixed-arithmetic-small}.  On compact sets,
multilinearity of the locally bounded Chern products and
\eqref{eq:curvature-cancel} give
\[
 \eps^{-k}c_1(\bN_\eps)^m
 \longrightarrow \binom{m}{g}\omega^g\wedge(\ddc|u|_w)^k,
 \qquad |u|_w:=\sum_\chi w_\chi|u_\chi|.
\]
Thus the integral of every such $\varphi$ against the limiting positive
measure is zero.  Its support is contained in $u_\psi=0$.

Include $\chi_1,\ldots,\chi_k$, all coordinate characters, and the character
$\psi$ in the finite collection.  The expanded limiting measure is a positive
linear combination of positive mixed measures.  Every summand therefore has
the same support containment.  Take $\psi=e_i$ successively for each $i$.
This proves \eqref{eq:all-character-support}.  A dense set of algebraic
torsion specializations supplies a generic sequence of height zero by
\eqref{eq:torsion-height}, proving the final assertion.
\end{proof}

\section{Radial tangency and the small-points theorem}\label{sec:central-bound}
\subsection{Radial tangency and polarized monodromy}
The first implication is local and geometric once the support condition
is known.  Assume $m=g+k$, maximal abelian Betti rank, and $\ell(B)=m$.
On the common rank-regular analytic open, the abelian Betti fibres are
complex plaques of dimension $k$.  On a plaque the labels $a,b$ are
fixed and
\begin{equation}\label{eq:Fleaf}
 F_i=-2\pi i(t_i-w_i b)
\end{equation}
is holomorphic with $u_i=\operatorname{Re}F_i$.  The leafwise map $F$
has rank $k$ on this open.  Put
\[
 V_x=\im_\C(dF_x|_{\ker db_p})\subseteq\C^r.
\]
The arithmetic support theorem supplies the following lemma's support
hypothesis for a generic small sequence, including dense torsion.

\begin{lemma}\label{lem:radial}
In this local period setting, assume that every mixed character measure
$\mu_{\chi_1,\ldots,\chi_k}$ is supported on $u=0$.
At every point of the common rank-regular analytic open,
\begin{equation}\label{eq:radial}
 u(x)\in V_x.
\end{equation}
\end{lemma}
\begin{proof}
For $k=0$, the support statement says that the smooth measure $\omega^g$ is
supported on $u=0$.  It is strictly positive on the maximal-rank locus, so
$u=0$ there, as required.

Suppose $k>0$ and $u(x)\notin V_x$.  Then $u(x)\ne0$ and
$V_x\cap\C u(x)=0$.  The complexification of the real annihilator of $u(x)$
surjects onto $V_x^*$.  Choose real linear forms
$\lambda_1,\ldots,\lambda_k$ annihilating $u(x)$ whose complex restrictions
to $V_x$ are independent.  The map
\[
 (\lambda_1u,\ldots,\lambda_ku)
\]
on the leaf is a real submersion at $x$: the corresponding holomorphic
differentials are a complex cotangent basis, whose real parts are real-linearly
independent.  Approximate the $\lambda_j$ by rational linear forms.
The real implicit-function theorem gives nearby common zeros for the
approximating forms, at which the complex differentials remain independent
and $u\ne0$.

Clear the denominators to obtain integral characters $\chi_j$.  At the
resulting regular common zero, the current in \eqref{eq:mixed-measure} has
positive local density on that common zero set.  Explicitly, regularize the
absolute values as in \cref{lem:positive-measures}; modulo the kernel of
$\omega$, the positive factors are the independent
$du_{\chi_j}\wedge d^cu_{\chi_j}$.  To see the normalization and the support directly, on the leaf
use holomorphic coordinates
$\zeta_j=x_j+iy_j=\chi_jF$ at the regular common zero.  Our convention
for $d^c$ gives
\[
  \ddc|x_j|=\frac1\pi\delta_0(x_j)\,dx_j\wedge dy_j.
\]
Their product is $\pi^{-k}\prod_j\delta_0(x_j)$ times
$\bigwedge_j(dx_j\wedge dy_j)$, a positive transverse measure in
the $k$ real equations.  The implicit-function
theorem also continues these zeros over a transverse neighborhood of the
abelian Betti variables, on which $\omega^g$ is positive.  Hence the point
belongs to the support of the full measure on $B$, not merely to a slice of
zero transverse mass.  This contradicts
\eqref{eq:all-character-support}, since $u\ne0$.
\end{proof}

Radial tangency has a stronger local consequence than the numerical
bound needed below.  It makes the leafwise image affine.  This
observation is purely holomorphic and does not assert rationality of
the resulting affine space.

\begin{proposition}\label{prop:radial-affine}
Let $F:Y\to\C^r$ be holomorphic of constant rank $k$ on a complex
manifold.  Suppose that
$\operatorname{Re}F(y)\in\im_\C dF_y$ for every $y$.
Locally there are a $k$-dimensional complex vector space
$W\subseteq\C^r$ with $\overline W=W$ and a vector $\eta\in\R^r$
such that the image of $F$ is open in $i\eta+W$.
\end{proposition}
\begin{proof}
Choose a connected coordinate neighborhood and a holomorphic matrix
$L(z)$ with kernel $\im dF_z$, which is a holomorphic subbundle by
constant rank.  The hypothesis says
\[
 L(z)\bigl(F(z)+\overline{F(z)}\bigr)=0.
\]
Expand in the local variables and their conjugates.  Vanishing of all
coefficients gives the polarized identity
\[
 L(z)\bigl(F(z)+\overline{F(w)}\bigr)=0
\]
for independent nearby $z,w$; the second factor is viewed on the
conjugate coordinate neighborhood.  Fix $w_0$ and set
$W=\operatorname{span}_\C\{F(w)-F(w_0)\}$.  Subtracting the polarized
identities gives $\overline W\subseteq\im dF_z$, so $\dim W\le k$.
Differentiating differences gives the reverse inclusion
$\im dF_z\subseteq W$.  Hence
$\im dF_z=W=\overline W$ and $\dim W=k$.
The image lies in $F(w_0)+W$; the original radial condition gives
$\operatorname{Re}F(w_0)\in W$.  Thus this translate equals
$i\operatorname{Im}F(w_0)+W$.  Constant rank makes the local image
open in it.  When $k=0$, the statement says that $F$ is a purely
imaginary constant.
\end{proof}

Applied to \eqref{eq:Fleaf}, this says that the central Betti image
$c=t-wb$ of each regular abelian Betti plaque is locally open in
$c_0+W$, with $c_0\in\R^r$ and $\overline W=W$.
For a proper direction $W$, the argument proves neither rationality
of $W$ nor a torsion condition on the translation.  It therefore does
not classify lower-rank plaques by algebraic subtori or by Ribet loci.

The next result isolates the geometric consequence of radial tangency.
Here we return to an arbitrary complex algebraic base; the mixed rank
need not equal the base dimension.  The polarization of real-analytic
identities follows \cite[\S5.2]{AndreCorvajaZannier2020} and
\cite[\S9]{KuehneRBC2023}.  We use it to isolate the action of the
central translations on the additional column; no arithmetic hypothesis
or functional-transcendence statement for derivatives enters this step.

\begin{proposition}[Polarized central bound]\label{prop:jet-bound}
In the algebraic semiabelian period setting over $\C$, suppose on a
nonempty rank-regular analytic open that
\[
  \rank_\C A=g,\qquad
  \rank_\C\begin{pmatrix}A\\D\end{pmatrix}=\ell,
  \qquad
  \binom{0}{u}\in\im\begin{pmatrix}A\\D\end{pmatrix}.
\]
Let $M$ be the connected algebraic monodromy group.  At a point $x$
in this open,
\begin{equation}\label{eq:geometric-central-bound}
  (M\cap U)_\C\subseteq D_x(\ker A_x),\qquad
  \dim(M\cap U)\le\ell-g.
\end{equation}
\end{proposition}
\begin{proof}
Choose an algebraic tangent frame on a smooth affine neighborhood
of $x$.  The augmented $(\ell+1)$-minors of
\begin{equation}\label{eq:augmented-column}
  \begin{pmatrix} A&0\\D&u\end{pmatrix}
\end{equation}
vanish.  Their entries are rational expressions in the holomorphic
period coordinates and their first derivatives, and in independent
conjugate period coordinates.  Specifically,
\[
  b=(\tau-\bar\tau)^{-1}(z-\bar z),\qquad
  u=\frac\pi i\bigl(t-\bar t-(w-\bar w)b\bigr).
\]
Multiply the minors by a sufficiently large power of
$\det(\tau-\bar\tau)$ to clear all denominators.

These are real-analytic identities.  Expand them in convergent
power series in local variables $s$ and $\bar s$ and replace
$\bar s$ by an independent complex variable.  Vanishing on the
real diagonal forces every coefficient to vanish, so the
polarized identities hold on a product neighborhood.
Freeze the second variable at $x$.  We obtain identities in the
holomorphic first jet of the period map.  They continue over the
connected universal cover.  The chosen algebraic tangent frame
may be meromorphic outside its initial open; this causes no
problem, since the identities continue as meromorphic identities,
and we evaluate only at lifts of $x$, where the frame is regular.

At deck translates of $x$, these first jets are the monodromy
translates of the original jet, while the frozen second-variable
values remain unchanged.  The action on period coordinates is
algebraic, and rational in the Siegel chart.  Its induced action
on first jets is rational by the chain rule.  After also clearing
the denominators of that action, the identities hold on the
Zariski closure of the monodromy orbit, hence on the $M_\C$-orbit.
This uses algebraic monodromy, not functional transcendence for
jets.

For $c\in(M\cap U)_\C$, central translation changes $t$ to $t+c$
and fixes $\tau,z,w$ and all first derivatives.  With the second
variable frozen, it therefore fixes $A,D$ and changes $u$ to
$u+(\pi/i)c$.  The determinant $\det(\tau-\bar\tau(x))$ stays
nonzero, and the matrix $\binom{A}{D}$ retains rank $\ell$.
The cleared minors consequently give both column containments
\[
  \binom{0}{u},\quad \binom{0}{u+(\pi/i)c}
       \ \in\im\begin{pmatrix}A\\D\end{pmatrix}.
\]
Subtracting them and using complex linearity yields
$c\in D(\ker A)$.

Finally, projection of $\im\binom{A}{D}$ onto $\im A$ is surjective
and has kernel identified with $D(\ker A)$.  Thus
\[
  \dim_\C D(\ker A)=\rank\binom{A}{D}-\rank A=\ell-g,
\]
which proves both assertions.
\end{proof}

The two geometric monodromy propositions have an immediate consequence
that does not use the arithmetic support theorem or the Betti-strata
construction.

\begin{corollary}\label{cor:geometric-full-rank}
Assume $K\subseteq P$, maximal abelian Betti rank $2g$, and at least
one torsion specialization of the marked section.  If on a nonempty
rank-regular analytic open
\[
 \binom0u\in\im\begin{pmatrix}A\\D\end{pmatrix},
\]
then $\ell=g+r$.
\end{corollary}
\begin{proof}
Proposition~\ref{prop:full-center} gives $M\cap U=U$.
Proposition~\ref{prop:jet-bound} gives
$r\le\ell-g$, while always $\ell\le g+r$.
\end{proof}

Combining the geometric first-jet statement with the support theorem
gives the arithmetic estimate needed for the dimension argument.

\begin{corollary}\label{cor:central-bound}
Suppose the marked family is defined over a number field, has a generic
sequence of algebraic points with $\htot\to0$, satisfies
$\rank_\R db_p=2g$, and has $\ell(B)=m=g+k$.
Then
\begin{equation}\label{eq:central-bound}
  \dim(M\cap U)\le k.
\end{equation}
\end{corollary}
\begin{proof}
Proposition~\ref{prop:all-character-support} and
Lemma~\ref{lem:radial} say that $u$ belongs to the image of the
leafwise differential of $F$.  Since $A$ vanishes precisely on
abelian Betti tangent directions and $dF=-2\pi iD$ there, this is
the column containment in Proposition~\ref{prop:jet-bound}.
That proposition gives $\dim(M\cap U)\le\ell-g=m-g=k$.
\end{proof}

\begin{remark}
Real central translations preserve an unpolarized unit-norm condition.
In the preceding proof the antiholomorphic period data are held fixed while
the holomorphic first jet is translated.  This is why a complex central
translation imposes a new column condition.  Without that polarization, the
central-monodromy conclusion would not follow.
\end{remark}

\subsection{Small points and torsion over algebraic numbers}\label{sec:qbar}
The arithmetic conclusions are proved over $\Qbar$.  All constructions
use finitely many coefficients and therefore descend to a number field.
The small-points argument below uses neither the Betti-strata theorem
nor the later passage to arbitrary complex coefficients.

We first record two elementary observations concerning the points to
which heights are applied.  If $F_0\subseteq\C$ is algebraically
closed and a marked family is defined over $F_0$, each fixed-order
locus $\{[n]h=0\}$ is $F_0$-defined and has Zariski-dense $F_0$-points
in every component.  Its base extension is exactly the corresponding
complex fixed-order locus.  Consequently complex torsion density for
$F_0$-defined data implies density of $F_0$-torsion points.  The same
statement holds with $n$ restricted to powers of a fixed prime.

Over $\Qbar$ there are only countably many proper closed subvarieties
of a fixed variety.  If the set of points of height less than $\eta$
is dense for every $\eta>0$, enumerate those subvarieties and, at
stage $n$, choose a point of height less than $1/n$ outside the first
$n$.  This gives a generic small sequence.  The construction also
works inside the set of points with torsion abelian projection.
Conversely, a generic small sequence makes every positive-height
sublevel set Zariski dense.  These observations justify the equivalent
gap formulations used below.

The constant-character obstruction of Section~\ref{sec:full-center}
can be excluded by small heights as well as by a torsion value.

\begin{lemma}\label{lem:small-center}
Let the marked family be defined over $\Qbar$, with $K\subseteq P$
and $\rank_\R db_p=2g$.  If it has a generic sequence $b_n$ with
$\htor(h(b_n))\to0$, then $M\cap U=U$.
\end{lemma}
\begin{proof}
If the center were deficient, Lemma~\ref{lem:constant-character}
would give, after fixed refinements, a nonzero multiplicative quotient
whose marked value is a constant $c\in\Qbar^*$ that is not a root
of unity.  A generic sequence lifts to the fixed cover and avoids any
fixed proper exceptional subset.  Exact norm functoriality
\eqref{eq:exact-norm-functoriality}, including the fixed pushout integer,
gives
\[
 2h_{\mathrm{Weil}}(c)
   =\sum_vn_v|\log|c|_v|
   \le C\htor(h(b_n))\longrightarrow0.
\]
Kronecker's theorem makes $c$ a root of unity, a contradiction.
\end{proof}

The quotient argument is geometric; exact torsion was only one
property that its homomorphism could preserve.  The canonical-height
comparison supplies the preservation required for small points.

\begin{lemma}\label{lem:small-quotient}
Let $B\subseteq\MM_{g,r}$ be a smooth marked classifying image over
$\Qbar$ satisfying $K\subseteq P$ and possessing a generic sequence
with $\htot\to0$.  If $\ell<m$, the quotient of
Proposition~\ref{prop:rank-reduction} satisfies
\[
 d'=d-h_N,\qquad m'\le\ell-h_N,
\]
is relatively nonspecial, and has a generic sequence with
$\htot'\to0$.  If the original abelian Betti rank is $2g$, the
quotient has abelian Betti rank $2g'$.
\end{lemma}
\begin{proof}
The saturated quotient and its classifying image are defined over the
same algebraically closed field.  Lemma~\ref{lem:height-functorial}
preserves smallness under every fixed homomorphism and isogeny used in
the construction.  Genericity is preserved by a dominant image: the
inverse image of a proper closed subset of the image closure is
proper.  It is also preserved by choosing arbitrary lifts to a fixed
irreducible finite cover dominating the base, since the finite image
of a proper closed subset of that cover is proper.  All covers,
isogenies, and open restrictions are fixed before taking the sequence.

In flat homology frames, the abelian quotient induces a constant
surjective rational matrix $F$.  Its Betti maps satisfy
$b_{p'}=F b_p$ modulo a locally constant lattice term.  Hence
$db_{p'}=F\,db_p$, which is surjective if $db_p$ is surjective.
The quotient map factors through the new classifying image; since
that map is dominant, passage to the image retains the generic rank.
Relative nonspecialness and the two dimension inequalities are already
part of Proposition~\ref{prop:rank-reduction}.
\end{proof}

We can now prove the small-points theorem stated in the introduction.
This is the point at which the geometric quotient, arithmetic support,
and central-monodromy estimates meet.

\begin{proof}[Proof of Theorem~\ref{thm:smallrank}]
Apply the classifying-image construction of
Proposition~\ref{prop:universal-reduction}, replacing torsion
preservation by Lemma~\ref{lem:height-functorial}.  Finite refinements,
fixed isogenies, and passage to the smooth marked image preserve
relative nonspecialness and smallness.  They preserve the generic
mixed rank: isogenies identify the local foliations up to finite
lattices, and a dominant classifying map is generically submersive.
They also preserve maximal abelian rank.

Suppose that a counterexample with $\ell<d$ exists.  Among all counterexamples in universal-section form, choose one
with minimal base dimension $m$.
If $m=0$, a generic sequence is the constant sequence at the algebraic
base point; its height is zero.  The height-zero criterion in
Section~\ref{sec:canonical-height} makes the section torsion, contrary
to $K\subseteq P$ when $d>0$.  The case $d=0$ is immediate.
If $r=0$, the assumed abelian rank already gives $\ell=g=d$.
If $g=0$, split the torus by a fixed finite cover and let $Y$ be the
closure of $h(B)$ in $\Gm^r$.  It has dense small points for twice the
sum of the coordinate Weil heights.  The Bogomolov theorem for tori,
in the form of \cite[Proposition~21]{KuehneSmall2022}, makes $Y$ a
torsion translate of a subtorus.  Relative nonspecialness excludes a
proper such translate.  Thus $h$ is dominant onto $\Gm^r$, and its
logarithmic Betti map has real rank $2r$, again not a counterexample.
We may therefore assume $g,r>0$ for the remaining argument.

If $\ell<m$, Lemma~\ref{lem:small-quotient} gives another problem
with all the hypotheses and
\[
 m'\le\ell-h_N<d-h_N=d',\qquad m'<m.
\]
Its mixed rank is at most $m'$, so it is a smaller counterexample.
Thus $\ell=m<d$.  Maximal abelian rank gives $m=g+k$ with $k\ge0$.
Lemma~\ref{lem:small-center} gives $M\cap U=U$.
Proposition~\ref{prop:all-character-support} and
Lemma~\ref{lem:radial} give the radial column condition on the common
rank-regular analytic open.  The polarized first-jet bound then gives
\[
 r=\dim(M\cap U)\le\ell-g=m-g,
\]
contradicting $m<g+r$.  No positive leading intersection is being
assumed in this argument.

Finally, failure of a positive gap would give a generic small sequence
by the enumeration argument above.  The closure of a nondense sublevel
set is the required proper closed exceptional set.  The tautological
section proves the statement for a dominating subvariety.
\end{proof}

With maximal abelian rank assumed, the arithmetic condition can be
expressed entirely in terms of the mixed rank.

\begin{corollary}\label{cor:small-equivalence}
For a relatively nonspecial marked family over $\Qbar$ with maximal
abelian Betti rank, the following are equivalent: it has a generic
$\htot$-small sequence; its real mixed Betti rank is $2(g+r)$; its torsion
specializations are Zariski dense.  When these conditions hold, every
prime-primary torsion set is Zariski dense.
\end{corollary}
\begin{proof}
The first implication is Theorem~\ref{thm:smallrank}.  Full rank gives
primary density by Proposition~\ref{prop:primary-density}; the
fixed-order observation makes the algebraic torsion points dense.
Torsion has height zero and supplies a generic sequence.
\end{proof}

The abelian theorem supplies precisely the missing rank hypothesis when
the marked points are torsion.  This recovers the arithmetic part of
the relative Manin--Mumford theorem without the later strata argument.

\begin{proposition}\label{prop:Qbar-completion}
The conclusion of Theorem~\ref{thm:main} holds for data defined over
$\Qbar$.
\end{proposition}
\begin{proof}
It suffices to treat a relatively nonspecial universal-section problem.
The fixed-order observation supplies a generic sequence of algebraic
torsion points.  The abelian projection generates the geometric generic
fibre by Lemma~\ref{lem:full-relative}; Gao--Habegger's theorem gives
$\rank_\R db_p=2g$.  For $g=0$ this condition is empty.  Torsion
points have $\htot=0$, so Theorem~\ref{thm:smallrank} gives
$\dim B\ge g+r$.  Since $\dim B\le\dim X$, this proves the claim.
\end{proof}

The next same-field implication is needed in the transcendence-degree
induction.  It is a consequence of the dimension theorem over the
\emph{same} field, not an application of an unproved complex version.

\begin{corollary}\label{cor:rank-criterion}
Over an algebraically closed subfield of $\C$ on which the dimension
statement of Theorem~\ref{thm:main} holds, a relatively nonspecial
torsion-dense universal section satisfies
\begin{equation}\label{eq:full-mixed-rank}
 \ell(B)=d.
\end{equation}
\end{corollary}
\begin{proof}
The dimension statement gives $m\ge d$.  If $\ell<d$, then
$\ell<m$ and Proposition~\ref{prop:rank-reduction} yields
$m'\le\ell-h_N<d-h_N=d'$, contradicting that same dimension
statement for the quotient.  The quotient and its classifying image
remain over the given algebraically closed field by
Proposition~\ref{prop:saturated}.  Since $\ell\le d$, equality follows.
\end{proof}

The converse rank-to-density implication is already available over
$\C$ from Proposition~\ref{prop:primary-density}.  The induction below
only needs the necessity implication recorded in this corollary.

\section{Mixed Betti strata and complex coefficients}\label{sec:strata}
This section is purely geometric.  No torsion-density hypothesis and no
arithmetic intersection theory are used.  We work with a smooth irreducible
locally closed algebraic subvariety $B$ of $\MM_{g,r}$, and put $m=\dim B$.
For an integer $t\in\Z_{\ge0}$, define
\begin{equation}\label{eq:betti-stratum}
 B^{\mathrm{Betti}}(t)=
 \{b\in B:\dim_b(B\cap\mathcal L_b^{\mathrm{Betti}})\ge t\}.
\end{equation}
Here the intersection is taken after a local period lift, and local dimension
is the maximum of the dimensions of its local irreducible components.
All dimension thresholds in this section are integral.  Extending the
notation to real $t\ge0$ would simply give
$B^{\mathrm{Betti}}(t)=B^{\mathrm{Betti}}(\lceil t\rceil)$.
The lattice invariance of the foliation makes the definition independent
of the lift.  We may make a finite level refinement to arrange connected
monodromy: the condition is local for an \'etale map, and closedness descends
under a finite surjective map.

Our finiteness argument concerns a single algebraic family of equations
for leaves.  It does not require a finite-family theorem for the closures
of every weakly optimal subvariety.  The input is instead the
flat-leaf intersection theorem of Baldi--Urbanik
\cite[Propositions~3.10 and~8.1, including the proof of the latter]{BaldiUrbanik2025}.
We give the translation to the present period coordinates and the
normal-orbit dimension calculation explicitly.

\subsection{The period torsor and algebraic incidence}
The integral admissible variation has sparse algebraic monodromy by
\cite[Lemma~4.14]{BaldiUrbanik2025}.  This checks the monodromy-class
hypothesis of the flat-leaf theorem used below.

Write $H$ for the connected algebraic monodromy of the marked variation
on $B$.  Its algebraic period torsor is denoted by
\[
 \pi_{\mathrm{per}}:\mathscr P\longrightarrow B.
\]
It has structure group $H$, a regular-singular flat connection, and an
algebraic evaluation map
\[
 \rho:\mathscr P\longrightarrow\check{\DD}^{\,0},
\]
where $\check{\DD}^{\,0}$ is the complex algebraic monodromy orbit in
the compact dual.  A local flat section determined by a marking evaluates
to the local period map.  These constructions, including algebraicity of
the Hodge filtration and of $\rho$, are recalled in
\cite[\S\S2.2--2.3 and~4.6]{BaldiUrbanik2025}.
We use $\mathscr P$ for the torsor to distinguish it from the
Mumford--Tate group $P$.

In the affine period chart of Section~\ref{sec:betti}, introduce
\emph{complex} parameters
\[
 \lambda=(\alpha,\beta,\gamma)\in
 \mathcal Y:=\C^g\times\C^g\times\C^r
\]
and the algebraic equations
\begin{equation}\label{eq:complexified-leaves}
 \Lambda_\lambda:\qquad
 z=\alpha+\tau\beta,\qquad t=\gamma+w\beta.
\end{equation}
Actual Betti leaves correspond to $\alpha,\beta\in\R^g$, with
$\gamma\in\C^r$.  The larger complex parameter space is used only to
obtain an algebraic family.  We never identify an arbitrary complex
parameter with an actual Betti leaf.

Let $\mathscr P^\circ$ be the inverse image of this chart.  The incidence
$\{(x,\lambda):\rho(x)\in\Lambda_\lambda\}$ is closed in
$\mathscr P^\circ\times\mathcal Y$.  Take its reduced Zariski closure
\begin{equation}\label{eq:leaf-incidence}
 \mathcal Z\subseteq\mathscr P\times\mathcal Y,
 \qquad f:\mathcal Z\longrightarrow\mathcal Y.
\end{equation}
Its restriction to $\mathscr P^\circ\times\mathcal Y$ is exactly the
original incidence, including on each fibre.  Thus taking this closure
introduces no extra germ at a point in the period chart.
After a finite algebraic stratification of the image of $f$ and omission
of empty fibres, this is a finite collection of algebraic families of
closed subvarieties of the \emph{whole} torsor $\mathscr P$.
All arguments below are applied separately to these finitely many
families.  Equivalently, one may use a finite affine atlas in the compact
dual.  It would not be legitimate to regard the open
$\mathscr P^\circ$ itself as a principal $H$-torsor.

Denote a horizontal flat leaf in $\mathscr P$ by $\mathscr L_x$.
For a fixed family in \eqref{eq:leaf-incidence} and an integer
$e\in\Z_{\ge0}$, set
\begin{equation}\label{eq:flat-dimension-locus}
 \mathcal Z(f,e)=
 \{(x,\lambda)\in\mathcal Z:
       \dim_x(\mathscr L_x\cap\mathcal Z_\lambda)\ge e\}.
\end{equation}
The finiteness used below comes from irreducible components of these
algebraic loci, not from an identification of different notions of
optimality.

\begin{lemma}\label{lem:flat-dimension-closed}
For every integer $e\ge0$, the locus $\mathcal Z(f,e)$ is Zariski closed in
$\mathcal Z$.  In particular it has finitely many irreducible components.
\end{lemma}
\begin{proof}
Algebraic constructibility is
\cite[Proposition~3.10]{BaldiUrbanik2025}.  To check closedness, choose
a holomorphic foliation box for the regular flat connection.  In that
box the horizontal leaves are the fibres of a holomorphic submersion
$\theta$.  On the incidence variety the map
\[
 (x,\lambda)\longmapsto(\theta(x),\lambda)
\]
is holomorphic, and its local fibre dimension is the dimension in
\eqref{eq:flat-dimension-locus}.  Upper semicontinuity of local fibre
dimension makes this locus analytically closed on the source.
An algebraically constructible subset of a complex algebraic variety
that is analytically closed is Zariski closed: a dense constructible
subset of each irreducible component of its Zariski closure contains a
Zariski open dense subset, which is also analytically dense in that
component.  This proves the assertion.
\end{proof}

\subsection{Normal-quotient loci and closed Betti strata}
We record the part of Baldi--Urbanik's theorem that will be used.
The assertion about the generic Mumford--Tate group is obtained from
the very-general-point choice in their proof; normality in monodromy
alone would not imply it.

\begin{lemma}\label{lem:flat-leaf-envelope}
Fix an integer $e\ge0$ and an irreducible component $R$ of
$\mathcal Z(f,e)$ that contributes
irreducible leaf-intersection germs of dimension exactly $e$, and put
\[
 T_R=\ol{(\pi_{\mathrm{per}}\circ\operatorname{pr}_1)(R)}^{\Zar}
       \subseteq B,
\]
where $\operatorname{pr}_1:\mathscr P\times\mathcal Y\to\mathscr P$
is the first projection.
There is a connected rational group $N_R$, normal in the generic
Mumford--Tate group $P_R$ of the variation on $T_R$, with the following
property.  For each such germ
\[
 \mathscr U=(\mathscr L_x\cap\mathcal Z_\lambda,x),
 \qquad (x,\lambda)\in R,
\]
there is, after the permitted finite cover, a subvariety
$Y\subseteq T_R$ weakly special for the restricted variation, and a
flat $N_R$-subtorsor $\mathscr P_{x,\lambda}$
over $Y$ containing $\mathscr U$, such that
\begin{equation}\label{eq:BU-envelope-inequality}
 \dim_x(\mathcal Z_\lambda\cap\mathscr P_{x,\lambda})
       \ge \dim N_R+e.
\end{equation}
The quotient period map by $N_R$ is constant on $Y$.
Here groups on $T_R$ or $Y$ are computed on a smooth dense open of their
normalizations and transported in the rational local system.  When a
finite cover or normalization is needed, $Y$ and its base point are
read on that space.  We choose a local irreducible branch that realizes
the dimension bound in \eqref{eq:BU-envelope-inequality}; only existence
of such a branch is asserted.  Passing to its finite image preserves
its dimension.
\end{lemma}
\begin{proof}
Apply \cite[Proposition~8.1]{BaldiUrbanik2025}.  Its part~(i) supplies
an $N_R$-subtorsor over a weakly special $Y$, with connected monodromy
contained in $N_R$.  In the construction on pages~43--44, $Y$ is an
irreducible component of a fibre of the quotient period map after a
finite cover.  The subtorsor is obtained by continuing the given flat
leaf and saturating by $N_R$; in particular it is compatible with the
flat connection.  Its dimension is $\dim Y+\dim N_R$, whereas a
horizontal leaf in it has dimension $\dim Y$.  Thus their equation~(8.1)
is exactly \eqref{eq:BU-envelope-inequality}.

For completeness, the group needed here has more normality than the
bare containment $N_R\lhd H_{T_R}$ displayed in that proposition.
In its proof the group is chosen as the connected monodromy of the
weakly special closure of a very general leaf-intersection germ over
$R$.  The exceptional set explicitly includes the inverse image of
the proper Hodge locus of $T_R$.  The chosen germ therefore passes
through a Hodge-generic point of $T_R$.  Its weakly special closure
has generic Mumford--Tate group $P_R$.  Andr\'e's normality theorem,
including the first part of its proof
\cite[Theorem~1 and its proof, pp.~10--11]{Andre1992}, gives
\begin{equation}\label{eq:full-MT-normality}
 N_R\lhd P_R,\qquad N_R\subseteq P_R^{\mathrm{der}}.
\end{equation}
The first assertion is the normality needed for the quotient datum;
it is not inferred merely from normality in the derived group.
The same $N_R$ is used for every germ in the component $R$, as established
in the cited proof.  Changes of marking transport these groups in the
rational local system; arithmetic translates give the same images in the
mixed Shimura quotient.  This proves the additional normality assertion
and the compatibility with the normal quotient datum used below.
\end{proof}

Only components meeting the locus of irreducible germs of exact dimension
$e$ are needed.  At such a point the generic local intersection dimension
on its component is also $e$, by Lemma~\ref{lem:flat-dimension-closed}.
Since $1\le e\le m$, the construction produces finitely many groups and
associated special domains.  It imposes no finiteness assertion on the
set of individual algebraic subvarieties $Y$.

The next calculation is where the weight-$-2$ directions enter.  It
subtracts the dimension of the period-frame stabilizer before applying
the normal-orbit rank.

\begin{lemma}\label{lem:torsor-leaf-dimension}
Suppose the germ in Lemma~\ref{lem:flat-leaf-envelope} comes from an
actual Betti leaf.  Require that $x$ be the frame determined by a local
rational marking of the given variation, and transport $N_R$ in that
same marking.  Put $b=\pi_{\mathrm{per}}(x)$.  If a finite cover or
normalization is used, choose a lift and a local irreducible branch
realizing \eqref{eq:BU-envelope-inequality}, and take all local
dimensions on that branch and its base image.  No assertion about
every lift is intended.  Let $O$ be the normal $N_R$-orbit through
its period point and let
\[
 h_R=\frac12\dim_\Q V_{p,N_R}+\dim U_{N_R}.
\]
Then the subvariety $Y$ in that lemma satisfies
\begin{equation}\label{eq:envelope-base-dimension}
 \dim_bY\ge h_R+e.
\end{equation}
Consequently the local fibre through $b$ of the normal quotient on
$B$ intersected with the corresponding special locus has dimension
at least $h_R+e$.
\end{lemma}
\begin{proof}
The marking hypothesis identifies the transported rational subgroup and
the actual Betti equations in one period chart.  An arbitrary complex
frame is not being treated as a rational marking.  Work on the chosen
local branch throughout; passage to a finite normalization does not
change its dimension.  Write $N=N_R$, $F=\rho(x)$ and
$\check O=N(\C)F$.  In a flat marking let $F(y)$ denote the nearby
period point over $y\in Y$.  Constancy of the quotient period map puts
all these points in the same orbit $\check O$.  The homogeneous map
$N(\C)\to\check O$, $a\mapsto aF$, has local holomorphic sections;
choose one to write $F(y)=a(y)F$ near the selected point.  Evaluation
on the flat subtorsor initially has the form $(y,n)\mapsto nF(y)$.
Replacing the group coordinate $n$ by $na(y)$ makes it
\[
 (y,n)\longmapsto nF.
\]
A local section of the homogeneous map now identifies its fibres with
the product of $Y$ and $\operatorname{Stab}_N(F)$.  The latter has
dimension $\dim N-\dim\check O$.  The construction works after
restriction to the chosen local branch even when $Y$ is singular.
Therefore
\begin{align}\label{eq:torsor-dimension-identity}
 \dim_x(\mathcal Z_\lambda\cap\mathscr P_{x,\lambda})
   &=\dim_bY+\dim N_R-\dim\check O
       +\dim_F(\Lambda_\lambda\cap\check O)\notag\\
   &=\dim_bY+\dim N_R-h_R.
\end{align}
Indeed $O$ is open in $\check O$ at $F$, and
Lemma~\ref{lem:orbit-rank} computes the intersection with this
\emph{actual} Betti leaf to have dimension $\dim\check O-h_R$.
All calculations take place in the period chart, so the closure used
in \eqref{eq:leaf-incidence} contributes no additional local component.
Comparing \eqref{eq:torsor-dimension-identity} with
\eqref{eq:BU-envelope-inequality} proves
\eqref{eq:envelope-base-dimension}.  Finally $Y$ is contained in the
normal-quotient fibre, which proves the last assertion.
\end{proof}

For each of the finitely many components just obtained, take the special
closure of $T_R$ in $\MM_{g,r}$, its datum $(P_j,\DD_j)$, and the group
$N_j=N_R$.  At a compatible neat level these give
\begin{equation}\label{eq:finite-quotient-data}
 f_j:S_j^{\dagger}\longrightarrow\MM_{g,r},\qquad
 q_j:S_j^{\dagger}\longrightarrow Q_j.
\end{equation}
The first map is finite onto its special image by
\cite[Proposition~3.8(a)]{Pink1990}.  The second is induced by the
normal quotient datum of \cite[Proposition~2.9]{Pink1990}; its
connected local fibres are the normal orbits.  Algebraic realization
and functoriality are those of \cite[Chapter~9, especially~9.24--9.25]{Pink1990}.  The finite source accounts for possible self-intersections
of the special image.  No quotient map on the whole of $\MM_{g,r}$ is
required.  Put
\[
 B_j=B\times_{\MM_{g,r}}S_j^{\dagger},\qquad
 \nu_j:B_j\longrightarrow B,\qquad
 \alpha_j=q_j\circ\operatorname{pr}_2,
\]
and, with $h_j$ as in Lemma~\ref{lem:orbit-rank}, define
\begin{equation}\label{eq:closed-certificates}
 E_j(t)=\nu_j\{x\in B_j:
       \dim_x\alpha_j^{-1}(\alpha_j(x))\ge h_j+t\}.
\end{equation}
The final fibre-dimension conditions are imposed on the finite
pullbacks $B_j$ themselves; the smooth opens used to compute the
generic groups do not restrict the points tested by these closed
conditions.
Local fibre dimension is upper semicontinuous on the source of a
finite-type algebraic morphism.  Each locus inside braces is closed,
and its image $E_j(t)$ is closed because $\nu_j$ is finite.

These finitely many loci suffice for the desired geometric conclusion.

\begin{proposition}[Mixed Betti strata]\label{prop:betti-strata}
For every integer $t\ge0$, the subset $B^{\mathrm{Betti}}(t)$ is Zariski closed
in $B$.  More precisely, for $1\le t\le m$ the finite quotient data
above satisfy
\begin{equation}\label{eq:finite-strata-equality}
 B^{\mathrm{Betti}}(t)=\bigcup_jE_j(t).
\end{equation}
\end{proposition}
\begin{proof}
The cases $t=0$ and $t>m$ are immediate.  Suppose $1\le t\le m$.
First let a point $b$ have an irreducible local Betti-leaf intersection
of dimension $e\ge t$.  Lift $b$ to a period frame $x$ on the flat
leaf corresponding to its marking, and use its actual parameters
$\lambda=(a,b_{\mathrm{Betti}},c)$ in
\eqref{eq:complexified-leaves}.  Projection from the horizontal leaf
is a local analytic isomorphism onto the base.  Thus
$(\mathscr L_x\cap\mathcal Z_\lambda,x)$ is irreducible of dimension
$e$, and belongs to a component of $\mathcal Z(f,e)$.
Lemmas~\ref{lem:flat-leaf-envelope} and~\ref{lem:torsor-leaf-dimension}
give a dimension-attaining lift and hence membership in one of the
$E_j(t)$.  This is an existence statement about a suitable branch,
which is exactly what the definition of $E_j(t)$ requires.

If the local leaf intersection at $b$ is reducible, choose a component
of maximal dimension $e\ge t$.  At nearby general smooth points of
this component, away from its intersections with the other local
components, the reduced leaf-intersection germ is irreducible of
dimension $e$.  Those points belong to the finite union of the closed
sets $E_j(t)$ by the preceding paragraph.  Since that union is also
analytically closed, it contains $b$.  This proves the forward inclusion
without assuming local irreducibility at every point of the stratum.

Conversely, a point of $E_j(t)$ has a lift at which a component of the
intersection of $B_j$ with a quotient fibre has dimension at least
$h_j+t$.  Locally that fibre is a normal orbit $O$.  In $O$ the
Betti leaf through the point is a complex submanifold of codimension
$h_j$, by Lemma~\ref{lem:orbit-rank}.  The analytic intersection
dimension inequality therefore gives a leaf-intersection germ of
dimension at least $t$.  The finite special-domain map is locally the
map of the corresponding period subdomain, so this germ maps to the
Betti-leaf germ in $B$ with the same dimension.  Hence
$E_j(t)\subseteq B^{\mathrm{Betti}}(t)$.
This proves \eqref{eq:finite-strata-equality}, and its right side is
Zariski closed.
\end{proof}

\begin{remark}
Weak optimality and maximality among all monodromically atypical
subvarieties are different conditions.  Accordingly, the preceding
proof uses neither an identification of those conditions nor an
unrestricted exact-fibre interpretation of
\cite[Lemma~7.5]{BaldiUrbanik2025}.
It uses the fixed algebraic family \eqref{eq:complexified-leaves},
Proposition~8.1 of that paper, and the stabilizer subtraction
\eqref{eq:torsor-dimension-identity}.
Gao--Zhang's \cite[Proposition~4.7]{GaoZhang2026} states a
weakly-optimal finiteness result in their weight-$-1,0$ setting;
we do not cite it as a theorem for weight $-2$.
For the distinction between structural properties of mixed Shimura
varieties and weak finiteness, compare
\cite[\S9, Theorem~9.6 and the following remark in the cited author
version]{BarroeroDill2025}.  Our argument establishes the finite
closed loci needed for Betti strata, not that broader weak-finiteness
statement.
\end{remark}

\subsection{Descent in transcendence degree}\label{sec:complex}
We prove the dimension statement and its torsion-to-rank consequence together by
induction on transcendence degree over $\Qbar$.  The pattern follows the
abelian argument in \cite[Section~10]{GaoHabegger2026}; the geometric input
needed here is \cref{prop:betti-strata}.

The following elementary descent lemma keeps track of the one dimension that can be gained when adjoining one transcendental parameter.

\begin{lemma}\label{lem:field-closure}
Let $K\subseteq L\subseteq\C$ be algebraically closed fields with
$\trdeg_KL=1$, let $V/K$ be a variety, and let $X\subseteq V_L$ be
a \emph{closed} irreducible $L$-subvariety.
Let $\mathfrak X\subseteq V$ be its smallest $K$-defined closed
algebraic closure.  Then
\[
 \dim\mathfrak X\le\dim X+1.
\]
If equality of dimensions $\dim\mathfrak X=\dim X$ holds, then $X$ is
$K$-defined.  The $K$-closure of $X(L)\cap\mathfrak X(K)$ remains contained
in $X$ after extension to $L$.
\end{lemma}
\begin{proof}
Only finitely many coefficients define $X$.  They lie in a finite extension
of $K(t)$, so spreading over a curve and taking the image closure proves the
first inequality.  Equivalently, contract the prime ideal to a coordinate
ring over $K$ and apply the transcendence-degree formula.
Since $K$ is algebraically closed, $\mathfrak X$ is geometrically irreducible.
If the dimensions are equal, $X$ is the entire base extension of
$\mathfrak X$.

For the last assertion, on an affine chart write a polynomial vanishing on
$X$ as a finite sum of $K$-polynomials with coefficients linearly independent
over $K$.  At a $K$-point of $X$, every coefficient polynomial vanishes.
Therefore each vanishes on the $K$-closure of these points; the original
polynomial vanishes on its base extension.  Apply this to all defining
polynomials of $X$.
\end{proof}

Torsion loci of fixed order are algebraic over the smaller field.  Their one-dimensional closures supply the leaves needed in the induction.

\begin{proposition}\label{prop:field-step}
Suppose the dimension statement of \cref{thm:main} holds over an algebraically
closed $K\subseteq\C$.  Then it holds over every algebraically closed
$L\subseteq\C$ with $\trdeg_KL=1$.
\end{proposition}
\begin{proof}
By \cref{cor:rank-criterion}, the rank statement holds over $K$ as well.
Apply universal reduction, but at this stage retain the \emph{closed}
classifying image $X\subseteq\MM_{g,r,L}$ rather than just a smooth
locally closed open.  It is a relatively nonspecial torsion-dense
counterexample, of dimension $m<d$: taking its closure changes neither
dimension nor the generic full-relative condition.  Smooth opens will
be used only later to discuss Betti ranks.  Let $\mathfrak X$ be its smallest
$K$-defined closure.  Lemma~\ref{lem:field-closure} and the theorem over $K$
show that
\begin{equation}\label{eq:closure-dim}
 \dim\mathfrak X=m+1.
\end{equation}
Indeed, if it had dimension $m$, then $X$ would be $K$-defined and its
$K$-torsion points would be dense, because each fixed-order torsion locus is
$K$-defined.

All universal special subvarieties are defined over $\Qbar\subseteq K$ by
the canonical models of mixed Shimura subdata.  Thus $X$ and $\mathfrak X$ have
the same special closure, as do their classifying-base images.  In particular
$\mathfrak X$ still satisfies the full-relative-group condition.

Its $K$-torsion points are Zariski dense.  To see this, let $Z$ be their
$K$-closure.  Each fixed-order locus
$\mathfrak X\cap\MM_{g,r}[n]$ is $K$-defined and has dense $K$-points in all
its components.  Hence $Z_L$ contains every torsion point of $X$.
Density implies $X\subseteq Z_L$, and minimality of $\mathfrak X$ implies
$Z=\mathfrak X$.

The theorem over $K$ and \eqref{eq:closure-dim} force
\[
 \dim\mathfrak X=d,\qquad m=d-1.
\]
The rank statement over $K$ therefore gives
$\ell(\mathfrak X)=d$ generically on its smooth locus.

For each torsion point $x\in X(L)$, its $K$-closure has dimension zero or
one and is contained in one fixed-order torsion locus.  The zero-dimensional
closures are $K$-points.  Their $K$-closure has dimension at most $m$ by
\cref{lem:field-closure}, so it is proper in $\mathfrak X$.
It follows that the one-dimensional closures of the remaining torsion points
form a Zariski-dense union of curves in $\mathfrak X$; otherwise their closure,
together with that proper $K$-point closure, would contradict the minimality
of $\mathfrak X$.

At smooth points these curves are locally contained in Betti leaves, since
each lies in a fixed-order torsion locus.  Hence
$\mathfrak X^{\mathrm{Betti}}(1)$ is Zariski dense on the smooth locus.
By \cref{prop:betti-strata} it is closed there, so it is the entire smooth
locus.  This contradicts $\ell(\mathfrak X)=d=\dim\mathfrak X$.
\end{proof}

\begin{proof}[Proof of Theorem~\ref{thm:main}]
Proposition~\ref{prop:Qbar-completion} gives the dimension statement over
$\Qbar$.  Proposition~\ref{prop:field-step}, followed at each stage by
\cref{cor:rank-criterion}, gives the statement over algebraically closed
subfields of $\C$ of any finite transcendence degree over $\Qbar$.

All equations, maps, and level choices defining a given complex problem use
only finitely many coefficients.  Let $L$ be the algebraic closure inside
$\C$ of the field generated by those coefficients over $\Qbar$.
If torsion is dense over $\C$, it is dense over $L$: the union of the
$L$-points of the $L$-defined fixed-order loci has the same algebraic closure
after base extension.  Nonspecialness is unchanged because the universal
special loci are defined over $\Qbar$.  Thus the statement over $L$ applies.
Finally descend the refinements and the universal reduction using
\cref{lem:descent-special,prop:universal-reduction}.
\end{proof}

\begin{proof}[Proof of Corollary~\ref{cor:main-rank}]
Suppose first that torsion is dense.  Pass to the tautological section
and its smooth marked classifying image.  Theorem~\ref{thm:main} and
Corollary~\ref{cor:rank-criterion}, now over $\C$, give full mixed
rank on that image.  The dominant classifying map is generically
submersive in characteristic zero, and finite refinements and
isogenies preserve the local foliations and their ranks.  Pulling back
therefore gives generic real rank $2(g+r)$ on the smooth locus of $X$.

Conversely suppose the mixed rank is $2(g+r)$.  Apply
Proposition~\ref{prop:primary-density} to the tautological section
\emph{in the original semiabelian family} over a smooth dense open of
$X$.  Only a finite \'etale torus-splitting cover is used in that
proposition; no primary-torsion reflection through an arbitrary
isogeny is needed.  It gives density of every prime-primary torsion
set even within the full-rank locus, and hence density of all torsion.
Density for a single prime implies density of all torsion; the other
implications then give the assertion for every prime.  This proves all
claims of the corollary.
\end{proof}

\section{Bogomolov transfer and toric height gaps}\label{sec:transfer}
All data in this section are defined over $\Qbar$.  The height
$\htot$ is the one in \eqref{eq:canonical-height}, with the fixed
refinement convention stated there.  The small-points theorem treats
the extension directions once the abelian Betti rank is maximal.
An abelian dimension theorem for all relevant quotients supplies that
rank and yields a transfer principle for moving semiabelian families.

\subsection{The abelian input and semiabelian transfer}
Let $\mathcal C$ be a class of abelian varieties over algebraically
closed extensions of $\Qbar$.  Assume that it is stable under isogenies
and abelian quotients, and that membership is preserved and reflected
by extension of algebraically closed fields.  The last condition
ensures that membership passes between a family and its classifying
image.

We say that $\mathcal C$ has the \emph{abelian Bogomolov dimension
property} if the following holds.  For every marked abelian family
whose geometric generic fibre belongs to $\mathcal C$, suppose that
no nonzero multiple of the marked section lies generically in a proper
abelian subvariety, even after a finite base change.  A generic
sequence on which its fibrewise canonical height tends to zero must
then force the dimension of the marked classifying image to be at
least the abelian relative dimension.  The general relative
Bogomolov formulation is \cite[Conjecture~1.2]{DGH2022}; here we
require only its indicated dimension consequence, for the class and
its quotients.

A deficient abelian rank would produce a quotient violating this same
dimension property.

\begin{lemma}\label{lem:abelian-bogomolov-rank}
Assume that $\mathcal C$ has the abelian Bogomolov dimension property.
A relatively nonspecial marked abelian family in $\mathcal C$ with a
generic small sequence has abelian Betti rank $2g$.
\end{lemma}
\begin{proof}
Pass to the marked abelian classifying image, of dimension $m$, and
write $a$ for half the real Betti rank.  Relative nonspecialness
implies the generation hypothesis by Lemma~\ref{lem:full-relative},
so the dimension property gives $m\ge g$.  If $a<g$, then $a<m$.
Apply Proposition~\ref{prop:rank-reduction} with toric rank zero.
The resulting quotient is relatively nonspecial and satisfies
\[
 g'=g-h_N,\qquad m'\le a-h_N<g'.
\]
Its canonical heights still tend to zero by
Lemma~\ref{lem:height-functorial}, and its generic sequence remains
generic by the dominant-image argument.  No maximal-rank assumption
is needed for this preservation of smallness.  The geometric generic
quotient and its classifying image belong to $\mathcal C$.  Applying
the same dimension property to that quotient is a contradiction.
\end{proof}

We can now pass from the abelian input to arbitrary toric extensions.
The condition on $\mathcal C$ is an explicit input, not a claim that
the general abelian relative Bogomolov conjecture is proved here.

\begin{theorem}[Abelian-to-semiabelian Bogomolov transfer]\label{thm:bogomolov-transfer}
Let $\mathcal C$ have the properties just specified and the abelian
Bogomolov dimension property.  Let $S/\Qbar$ be smooth, irreducible,
and quasi-projective, and let $\G/S$ be a semiabelian scheme of
relative dimension $g+r$ whose geometric generic abelian quotient
belongs to $\mathcal C$.  Let $X\subseteq\G$ be closed, irreducible,
dominating $S$, and relatively nonspecial.  The following conditions
are equivalent: $X$ has a generic sequence with $\htot\to0$; its
real mixed Betti rank is $2(g+r)$; its fibrewise torsion is Zariski
dense; and its prime-primary torsion is Zariski dense for every prime.
If the real mixed Betti rank is less than $2(g+r)$, there exist
$\eta>0$ and a proper closed $Z\subsetneq X$ such that
$\htot(x)\ge\eta$ for every $x\in X(\Qbar)\setminus Z$.
In particular this gap holds whenever $\dim X<g+r$.
\end{theorem}
\begin{proof}
Suppose first that $X$ has a generic $\htot$-small sequence.  Use its
tautological marked point over a smooth dense open.  The abelian marked
classifying image then has a generic sequence of canonical height
tending to zero.  The image of the full relative
translation group is the full abelian relative translation group;
thus the abelian problem is relatively nonspecial.
Lemma~\ref{lem:abelian-bogomolov-rank} gives full abelian Betti rank
on that image and hence on $X$, by generic submersivity of the
dominant classifying map.  Theorem~\ref{thm:smallrank} now gives the
mixed rank.  Proposition~\ref{prop:primary-density} gives primary
torsion density from full rank; the fixed-order algebraicity argument
in Section~\ref{sec:qbar} makes algebraic torsion dense.  These points
supply a generic sequence of height zero.  This closes the equivalence
without using complex-coefficient descent.  If the mixed rank is
deficient, the equivalence and the generic-sequence criterion give
a positive gap.  The dimension assertion follows because the
mixed rank is at most $2\dim X$.
Taking the closure of the exceptional sets also includes the singular
and previously removed loci of $X$.
\end{proof}

K\"uhne's theorem allows products of elliptic families over an arbitrary
base, including isotrivial factors, and explicitly extends to generic
isogenies.  This gives a large class of genuinely varying semiabelian
extensions.

\begin{corollary}\label{cor:elliptic-bogomolov}
Let $S/\Qbar$ be smooth, irreducible, and quasi-projective, let
$\G/S$ be a semiabelian scheme of relative dimension $g+r$, and let
$X\subseteq\G$ be closed, irreducible, dominating $S$, and relatively
nonspecial.  Suppose the geometric generic abelian quotient is
isogenous to a product of elliptic curves.  A generic sequence in
$X(\Qbar)$ with $\htot\to0$ exists if and only if the real mixed
Betti rank on $X$ is $2(g+r)$, equivalently if fibrewise torsion is
Zariski dense.  In that case $\dim X\ge g+r$, and every prime-primary
torsion set is dense even in the maximal-rank locus.  If the real
mixed Betti rank is less than $2(g+r)$, there are $\eta>0$ and a
proper closed $Z\subsetneq X$ such that $\htot(x)\ge\eta$ for all
$x\in X(\Qbar)\setminus Z$; in particular this holds if
$\dim X<g+r$.
The toric rank and base dimension are unrestricted; the semiabelian
extension need not split.
\end{corollary}
\begin{proof}
The abelian input is \cite[Theorem~1 and the following
paragraph]{KuehneRBC2023}.  To check the dimension consequence, apply
that theorem to a generated marked section over its marked classifying
image $B$.  A proper horizontal torsion coset containing this section
would, after a finite cover and multiplication killing its torsion
translate, place a multiple of the section in a proper abelian
subscheme.  The generation hypothesis excludes this.  The remaining
horizontal torsion coset is the whole family, in which the section
has codimension $g$.  If $\dim B<g$, the hypothesis of K\"uhne's
relative Bogomolov theorem is therefore satisfied, and a generic small
sequence is impossible.  This proves the required dimension property.
Different relatively ample symmetric height choices are comparable
without additive errors by Lemma~\ref{lem:height-functorial}.

An abelian quotient of a product of elliptic curves is isogenous to
a product of elliptic curves.  Indeed the images of the elliptic
factors generate the quotient; choose independent nonzero images
successively until their product maps isogenously onto it.  These
maps and the isogenies spread after a finite extension and shrinking.
Homomorphisms and abelian subvarieties of an abelian variety over an
algebraically closed field do not acquire new geometric morphisms
after an algebraically closed field extension; this follows from the
unramified relative Hom scheme, or from rigidity after descent to a
finite-type parameter space.  Thus the class has the required
field-extension property as well as quotient and isogeny stability.
The transfer theorem applies.
\end{proof}

The abelian quotient may instead be constant up to isogeny, in any
dimension.  Its extension parameters can still move.

\begin{corollary}\label{cor:isotrivial-bogomolov}
Let $S/\Qbar$ be smooth, irreducible, and quasi-projective, let
$\G/S$ be a semiabelian scheme of relative dimension $g+r$, and let
$X\subseteq\G$ be closed, irreducible, dominating $S$, and relatively
nonspecial.  Suppose that after a finite base change the abelian
quotient becomes isogenous to $A_0\times S$ for an abelian variety
$A_0/\Qbar$.  A generic sequence in $X(\Qbar)$ with $\htot\to0$
exists if and only if the real mixed Betti rank is $2(g+r)$,
equivalently if fibrewise torsion is Zariski dense.  Then
$\dim X\ge g+r$, and every prime-primary torsion set is dense even
in the maximal-rank locus.  If the real mixed Betti rank is less
than $2(g+r)$, there are $\eta>0$ and a proper closed $Z\subsetneq X$
such that $\htot(x)\ge\eta$ for every
$x\in X(\Qbar)\setminus Z$; in particular this holds if
$\dim X<g+r$.
The dimension of $A_0$ is arbitrary.  No isotriviality of the
semiabelian extension is assumed: its extension parameters may vary.
\end{corollary}
\begin{proof}
Dimitrov--Gao--Habegger prove the abelian relative Bogomolov statement
for isotrivial schemes in \cite[Proposition~4.1]{DGH2022}.
Fixed isogenies preserve smallness.  Geometric quotients of a constant
abelian variety are isogenous to constant quotients by rigidity of
homomorphisms; hence this class is stable under the operations in the
transfer theorem.

Equivalently, after making the quotient $A_0\times B$, let $Y$ be the
closure of $p(B)$ in $A_0$.  It has dense small points.  The classical
abelian Bogomolov theorem makes $Y$ a torsion translate of an abelian
subvariety \cite{Ullmo1998,Zhang1998}.  Generation excludes a proper
such translate, so $p$ is dominant and has abelian Betti rank $2g$.
Theorem~\ref{thm:smallrank} then applies directly.
\end{proof}

These corollaries concern arithmetic heights of specializations in
moving families.  They should be distinguished from the absolute
semiabelian Bogomolov and equidistribution theorem
\cite{KuehneSmall2022}, from geometric Bogomolov over function fields
\cite{LuoYu2025}, and from Hultberg's reduction over globally valued
fields and uniform gap principle for subvarieties of individual
semiabelian varieties \cite[Theorems~2--3]{Hultberg2026}.
The hypotheses on proper relative special loci, and the dimension of
the marked classifying image, are essential to the formulation here.
The constants obtained here are not asserted to be effective or
uniform in the family.

\subsection{Toric heights above abelian torsion}\label{sec:hybrid}
Even without an abelian small-points theorem, exact torsion in the
abelian quotient supplies full abelian rank.  The extension directions
then have a positive height obstruction.

\begin{theorem}[Toric height under deficient mixed Betti rank]
\label{thm:toric-gap}
Let $S/\Qbar$ be smooth, irreducible, and quasi-projective, let
$\G/S$ be a semiabelian scheme of relative dimension $d=g+r$, and
let $X\subseteq\G$ be closed, irreducible, dominating $S$, and
relatively nonspecial.  Suppose that
\[
  \ell(X):=\tfrac12\max_{x\in X^{\mathrm{sm}}}
                    \rank_\R d\beta_x<d.
\]
Then there are $\eta>0$ and a proper closed $Z\subsetneq X$ such that
\[
  x\in X(\Qbar)\setminus Z,\qquad p(x)\text{ torsion}
  \quad\Longrightarrow\quad \htor(x)\ge\eta.
\]
There is no structural restriction on the abelian quotient.
In particular, the conclusion holds whenever $\dim X<d$.
\end{theorem}
\begin{proof}
Suppose no such pair $(\eta,Z)$ exists.  For every $\eta>0$ the
points with torsion abelian projection and $\htor<\eta$ are then
Zariski dense.  The generic-sequence construction in
Section~\ref{sec:qbar} gives a generic sequence $x_n\in X(\Qbar)$
with torsion abelian projection and $\htor(x_n)\to0$.

Pass to the tautological section over a smooth dense open of $X$,
make the fixed finite refinements, and take the abelian marked
classifying image.  Genericity and smallness are preserved by
Lemma~\ref{lem:height-functorial} and the finite-cover and
dominant-image argument of Lemma~\ref{lem:small-quotient}.
The abelian image has dense torsion.  Moreover, the image of the
full relative translation group is the full abelian relative
translation group, so Lemma~\ref{lem:full-relative} gives the
generation hypothesis for that image.  If $g>0$, Gao--Habegger's
theorem \cite[Theorem~1.3]{GaoHabegger2026} gives abelian Betti rank
$2g$; for $g=0$ this condition is empty.  Generic submersivity of
the dominant classifying map pulls the rank back to the
marked family on $X$.

Since $\hab(x_n)=0$, one has
$\htot(x_n)=\htor(x_n)\to0$.  Theorem~\ref{thm:smallrank}
therefore gives $\ell(X)=d$, a contradiction.  The rank is unchanged
by the fixed refinements.  The exceptional set in the original $X$
is obtained by taking finite images and including the closures of
the singular and previously removed loci.  Finally,
$\ell(X)\le\dim X$, which proves the last assertion.
\end{proof}

The rank hypothesis is insensitive to redundant parameters: adjoining
base directions along which the marked family does not change can
increase $\dim X$ without increasing $\ell(X)$.  The rank-deficient
form therefore applies beyond the subcritical-dimensional case.

Multiplication by the abelian order converts this gap into an explicit
linear dependence on that order.  The coefficient remains ineffective
and depends on the marked family and the chosen height.

\begin{corollary}\label{cor:order-height}
Under the hypotheses of Theorem~\ref{thm:toric-gap}, let
$x\in X(\Qbar)\setminus Z$ have abelian projection of order $N$.
In the chosen toric coordinates write $[N]x=(t_1,\ldots,t_r)$.
Then
\begin{equation}\label{eq:order-height}
 \sum_{i=1}^r h_{\mathrm{Weil}}(t_i)\ge\frac\eta2N.
\end{equation}
If $X$ is a curve satisfying the stated rank hypothesis, the exceptional set is finite.
\end{corollary}
\begin{proof}
At the identity of the abelian quotient the rigidified Poincar\'e
norm is the ordinary absolute value.  Multiplication and the product
formula therefore give
\[
 N\htor(x)=\sum_{i,v}n_v|\log|t_i|_v|
          =2\sum_i h_{\mathrm{Weil}}(t_i).
\]
Apply the gap theorem.  A proper closed subset of an irreducible
algebraic curve is finite.
\end{proof}

\section{Generalized Jacobians and Pell equations}\label{sec:pell}
Generalized Jacobians give a concrete interpretation of the additional
toric obstruction.  Identifying pairs of points on a smooth curve
introduces a multiplicative gluing parameter for each node.  A divisor
can be torsion in the ordinary Jacobian while these gluing parameters
prevent torsion in the generalized Jacobian.  For hyperelliptic curves,
this is exactly the distinction between two polynomial Pell equations.
The generalized-Jacobian interpretation is classical
\cite[\S2.2, Proposition~1]{BertrandPell2015}; see also
\cite[Appendix~II]{BMPZ2016}.

\subsection{Torsion and gluing at the nodes}
Let $S/\Qbar$ be smooth and irreducible.  Let $Q_s(x)$ be a family of
monic squarefree polynomials of degree $2g+2$, and let $P_s(x)$ be a
family of squarefree polynomials of degree $r$, coprime to $Q_s$.
Shrink so that these conditions hold.  The smooth normalization
\[
 C_s:\quad y^2=Q_s(x)
\]
has two points at infinity, denoted $\infty_+,\infty_-$.
After a finite \'etale refinement, label the roots of $P_s$ and the
two points $a_i^+,a_i^-$ above each root.  Let $C'_s$ be the projective
curve obtained by identifying these pairs, leaving both infinities
smooth.  Its affine equation is
\[
 Y^2=P_s(x)^2Q_s(x).
\]
This specified completion avoids the additional singularities that a
plane homogenization could introduce at infinity.

The generalized Jacobians form the extension
\begin{equation}\label{eq:nodal-jacobian}
 0\longrightarrow\Gm^r\longrightarrow\mathcal J'
   \longrightarrow\mathcal J\longrightarrow0,
 \qquad \dim(\mathcal J'/S)=g+r.
\end{equation}
Indeed a line bundle on the normalization descends after choosing a
nonzero gluing scalar at each identified pair.  Equivalently, the
extension classes are $[a_i^+-a_i^-]$ in the dual Jacobian, up to the
choice of orientation of the pairs; their Poincar\'e extensions give
the algebraic family in \eqref{eq:nodal-jacobian}.
Let $\delta=[\infty_--\infty_+]$ be its marked section.

We give the elementary Pell criterion, including the node conditions,
so that the precise scope of the application is explicit.

\begin{lemma}\label{lem:pell}
Over an algebraically closed field of characteristic zero, the value
$\delta_s$ is torsion in $\mathcal J'_s$ if and only if
\begin{equation}\label{eq:nodal-pell}
 A(x)^2-P_s(x)^2Q_s(x)B(x)^2=1
\end{equation}
has a polynomial solution with $B\ne0$.
\end{lemma}
\begin{proof}
If $n\delta_s=0$ for $n>0$, there is a rational function $f$ on $C_s$
with divisor $n(\infty_--\infty_+)$ which descends as a unit through
every node.  Both $f$ and $f^{-1}$ are regular on the affine
normalization.  Its coordinate ring is $k[x,y]/(y^2-Q_s)$, so
$f=A(x)+yC(x)$ with polynomials $A,C$.  The hyperelliptic involution
$\iota$ makes $f\iota(f)=A^2-Q_sC^2$ a unit in $k[x]$, hence a
nonzero constant.  Rescale $f$ to make this constant one.

Descent at the node over a root $a_i$ of $P_s$ means
\[
 A(a_i)+\sqrt{Q_s(a_i)}C(a_i)
 =A(a_i)-\sqrt{Q_s(a_i)}C(a_i).
\]
Since $Q_s(a_i)\ne0$, one has $C(a_i)=0$.  Squarefreeness of $P_s$
gives $C=P_sB$.  The function $f$ is nonconstant, so $B\ne0$, and
\eqref{eq:nodal-pell} follows.

Conversely, a solution gives the unit $f=A+P_syB$ in the nodal affine
coordinate ring, with inverse $A-P_syB$.  Its divisor is supported
at the two smooth infinities and has degree zero, hence is
$n(\infty_--\infty_+)$ for an integer $n$.  Since $B\ne0$ and
$Q_s$ is not a square, $f$ is nonconstant, so $n\ne0$.  It descends
through the nodes and makes $\delta_s$ torsion in $\mathcal J'_s$.
\end{proof}

The dimension threshold now includes the number of nodes.  In the
range $g\le\dim S<g+r$, the abelian dimension inequality by itself
does not exclude density, whereas the semiabelian one does.

\begin{corollary}\label{cor:pell-nondensity}
If $\delta$ is relatively nonspecial and $\dim S<g+r$, then the
parameters for which \eqref{eq:nodal-pell} has a solution with
$B\ne0$ are not Zariski dense.  If $S$ is a curve and $g+r>1$,
there are only finitely many such parameters.
\end{corollary}
\begin{proof}
By Lemma~\ref{lem:pell}, these are torsion specializations of the
marked section of $\mathcal J'$.  Each fixed-order locus is
$\Qbar$-defined, so density of complex such parameters would imply
density of algebraic such parameters by the fixed-order observation
in Section~\ref{sec:qbar}.  Their abelian projections are torsion and
their toric heights vanish.  Apply Theorem~\ref{thm:toric-gap} to
the graph of $\delta$.
Alternatively the nondensity follows directly from
Theorem~\ref{thm:main}.  In dimension one a proper closed set is finite.
\end{proof}

More quantitatively, a unit on the normalization need not descend
through the nodes.  Its ratios of values measure exactly this failure.

\begin{corollary}\label{cor:gluing-height}
Under the hypotheses of Corollary~\ref{cor:pell-nondensity}, there are
$\eta>0$ and a proper closed exceptional subset of $S$ with the
following property.  For $s\in S(\Qbar)$, suppose the projection of $\delta_s$ to
$\mathcal J_s$ has order $N$, and choose $f_s$ on $C_s$ with
\[
 \divi(f_s)=N(\infty_--\infty_+).
\]
Put $\lambda_i(s)=f_s(a_i^+)/f_s(a_i^-)$.  Outside the exceptional
set,
\begin{equation}\label{eq:gluing-height}
 \sum_i h_{\mathrm{Weil}}(\lambda_i(s))\ge\frac\eta2N.
\end{equation}
\end{corollary}
\begin{proof}
The values at the nodes are nonzero and finite, and the ratios are
independent of scaling $f_s$.  Under the gluing description of
\eqref{eq:nodal-jacobian} they are the toric coordinates of
$[N]\delta_s$, up to inversion according to the chosen orientations.
Weil height is invariant under inversion.  Thus
Corollary~\ref{cor:order-height} gives \eqref{eq:gluing-height}.
\end{proof}

The degree of the map $f_s:C_s\to\mathbf P^1$ is $N$, since its
pole divisor is $N\infty_+$.  The inequality therefore compares the
arithmetic heights of gluing values with the function-field degree of
the unit on the normalization.  Relative nonspecialness of $\delta$
must be checked for any proposed explicit polynomial family.  It is
not automatic, and generic cyclic generation does not replace it.

\subsection{An explicit fixed nonsplit family}\label{sec:absolute-example}
The nonspecialness condition in the preceding application can be checked
without appealing to the torsion theorem itself.  The following example
keeps the generalized Jacobian fixed and nonsplit, while the infinity
divisor moves.  It makes the difference between the ordinary and nodal
Pell equations concrete.

Let
\[
 E_0:\quad y^2=x^3-x+1,\qquad c_0=(0,1),
\]
with origin the point at infinity.  On the smooth irreducible base
\[
 B=\{(t,v):v^2=t^3-t+1,\quad tv\ne0\}\subset E_0
\]
define the monic quartic and linear polynomials
\begin{equation}\label{eq:explicit-Pell-polynomials}
 \begin{split}
 Q_b(X)&=X^4+\frac{3t^2-1}{v^2}X^3
                   +\frac{3t}{v^2}X^2+\frac{1}{v^2}X,\\
 P_b(X)&=X+\frac1t,\qquad b=(t,v).
 \end{split}
\end{equation}
Let $C_b$ be the smooth completion of $Y^2=Q_b(X)$, and identify the
two points over $X=-1/t$ to form $C'_b$, keeping its infinities smooth.
Write $\delta_b=[\infty_--\infty_+]\in\operatorname{Pic}^0(C'_b)$.

\begin{proposition}\label{prop:explicit-pell}
For the family \eqref{eq:explicit-Pell-polynomials}, the polynomials
$Q_b$ are squarefree and coprime to $P_b$.  The generalized Jacobian
is a constant, nonsplit semiabelian surface, and its marked section
$\delta$ is relatively nonspecial.  Consequently the nodal equation
\[
 A(X)^2-P_b(X)^2Q_b(X)B_1(X)^2=1,\qquad B_1\ne0,
\]
has a polynomial solution for only finitely many $b\in B(\C)$.
In contrast, the ordinary equation
$A(X)^2-Q_b(X)C(X)^2=1$, with $C\ne0$, is solvable precisely when
$b$ is torsion on $E_0$, so it is solvable for infinitely many parameters.

There are also $\eta>0$ and a finite $Z\subset B$ such that the
following holds.  For $b\in B(\Qbar)\setminus Z$ such that $2b$ is
torsion of order $N$ on $E_0$, choose a rational function $f_b$ on $E_0$ with
$\divi(f_b)=N((-b)-(b))$.  Then
\begin{equation}\label{eq:explicit-gluing-gap}
 h_{\mathrm{Weil}}\!\left(\frac{f_b(c_0)}{f_b(-c_0)}\right)
       \ge\frac\eta2N.
\end{equation}
\end{proposition}
\begin{proof}
The substitution
\[
 X=\frac1{x-t},\qquad Y=\frac{y}{v(x-t)^2}
\]
gives $Y^2=Q_b(X)$, using $v^2=t^3-t+1$.  It identifies the smooth
completion with $E_0$ and takes the two infinities to $b$ and $-b$.
The three finite roots of $x^3-x+1$ and its point at infinity give
four distinct branch points after this change of variable; equivalently,
\[
 \operatorname{disc}_X(Q_b)=-\frac{23}{v^{12}}\ne0.
\]
The pair over $X=-1/t$ is exactly $c_0,-c_0$, since
\[
 Q_b(-1/t)=\frac{1}{v^2t^4}\ne0.
\]
Thus $C'_b$ is the fixed nodal curve $E_0/(c_0\sim-c_0)$, with a
varying choice of the two smooth points designated as infinities.
Its generalized Jacobian $G_0$ is an extension of $E_0$ by $\Gm$
with parameter $q_0=\pm2c_0$ under the principal polarization; see the
gluing description preceding Lemma~\ref{lem:pell}.  The sign depends
only on orientation and has no effect below.  The abelian projection
of the marked section is $p(b)=-2b$.

We first verify that $q_0$ is nontorsion.  The Weierstrass discriminant
of $E_0$ is $-16\cdot23$, and direct enumeration gives
\[
 \#E_0(\mathbf F_3)=7,\qquad \#E_0(\mathbf F_5)=8.
\]
Reduction at a good prime is injective on torsion of order prime to
that prime \cite[Proposition~VII.3.1(b)]{Silverman2009}.  A rational
point of prime order $\ell\notin\{3,5\}$ would force $\ell$ to divide
both $7$ and $8$.  Order $3$ is excluded by reduction at $5$, and
order $5$ by reduction at $3$.  Hence $E_0(\Q)$ has no nonzero torsion
and $c_0$, and therefore $q_0$, is nontorsion.  In particular $G_0$
is nonsplit, even up to isogeny.

It remains to establish relative nonspecialness; generic cyclic
generation alone would not suffice.  Use the rational realization of
$[\Z\xrightarrow{\delta}G_0\times B]$ and write $P,P_0,M,K,U,V$
as in Section~\ref{sec:universal}.  Since $G_0$ is constant, monodromy
acts trivially on its unmarked realization, and thus $M\subseteq K$.
The map $p:B\to E_0$, $b\mapsto-2b$, has abelian Betti rank $2$.
Lemma~\ref{lem:abelian-monodromy} therefore makes the $p$-projection of
$\Lie M$ equal to $V$.

The unipotent radical of $P_0$ is a rational Hodge subspace of
$\Hom(V,\Q(1))$.  It is nonzero: otherwise reductivity would split
the rational Hodge extension $H_1(G_0)$, making $q_0$ torsion by the
full faithfulness of the $1$-motive realization
\cite[\S10.1]{Deligne1974}.  Since $V$ is the simple weight-$-1$
Hodge structure of an elliptic curve, the dual Tate-twisted Hodge
structure is simple too.  Consequently that unipotent radical is
all of $\Hom(V,\Q(1))$.  Surjectivity $P\to P_0$ gives the full
$q$-projection in $\Gr^W_{-1}\Lie P$.

Andr\'e's theorem gives $M\lhd P$
\cite[Theorem~1 and its proof]{Andre1992}.  In a compatible weight
splitting, bracket an element of $\Lie M\subseteq\Lie K$, with
$p$-projection $v$, against an element of $\Lie P$ with
$q$-projection $q$.  Formula~\eqref{eq:heisenberg} gives $\pm q(v)$
in $\Lie(M\cap U)$.  The evaluation pairing is nondegenerate, so
$M\cap U=U$.  As $M\subseteq K$ also projects onto $V_p$, it follows
that $M=K$.  In particular $K\subseteq P$, and
Lemma~\ref{lem:full-relative} proves relative nonspecialness.  This
verification uses only the geometric monodromy and realization
arguments, not the relative torsion or small-points theorem.

Here $\dim B=g=1<g+r=2$.  Corollary~\ref{cor:pell-nondensity} and
Lemma~\ref{lem:pell} give the claimed finiteness for the nodal equation.
The same unit criterion without nodes identifies ordinary Pell
solvability with torsion of $p(b)=-2b$, equivalently of $b$.
There are infinitely many such $b$ in $B$.  Finally the gluing ratio
is $f_b(c_0)/f_b(-c_0)$, up to inversion.  Corollary~\ref{cor:gluing-height}
gives \eqref{eq:explicit-gluing-gap}; both values are finite and nonzero
because $b\ne\pm c_0$.
\end{proof}

Because $G_0$ is fixed, both the torsion finiteness and the height gap
in Proposition~\ref{prop:explicit-pell} also have absolute proofs.
Recording the latter distinguishes this worked example from the
moving-family statements of Section~\ref{sec:transfer}.

Put $C_\delta=\ol{\delta(B)}^{\Zar}\subset G_0$.  This is an
irreducible curve whose projection to $E_0$ is dominant.  It is not a
torsion translate of a one-dimensional connected subgroup.  Such a
subgroup with zero-dimensional abelian image would be a torus, whose
translates have constant abelian projection.  A subgroup with
positive-dimensional image would instead be an elliptic curve $E'$,
and its projection $f:E'\to E_0$ would be an isogeny.  Its inclusion
in $G_0$ would split $f^*G_0$, giving $f^\vee(q_0)=0$.  The finite
kernel of $f^\vee$ would then make $q_0$ torsion, a contradiction.
The absolute semiabelian Bogomolov theorem therefore gives a positive
canonical-height gap on $C_\delta$ outside a finite set
\cite[Proposition~21]{KuehneSmall2022}; the general absolute theorem
also has the earlier proof of \cite{DavidPhilippon2000}.

For the normalization of this deduction, let $M$ be the standard
boundary line bundle on the equivariant compactification of $G_0$,
restricting to $\OO_{\mathbf P^1}(0+\infty)$ on the compactified
torus.  Write $\widehat h_M$ for its canonical height and choose an
ample symmetric $\Theta$ on $E_0$.  The height used in the absolute
theorem is $\widehat h=\widehat h_M+\widehat h_\Theta\circ p$.
The canonical construction gives
\begin{equation}\label{eq:absolute-boundary-scaling}
 \widehat h_M([n]x)=n\widehat h_M(x),\qquad
 \widehat h_M(t)=2h_{\mathrm{Weil}}(t)\quad(t\in\Gm(\Qbar));
\end{equation}
see \cite[\S2.2, Lemmas~9--10, and \S2.5]{KuehneSmall2022}.
These are exact identities, not comparisons up to an additive constant.

If $2b$ is torsion of order $N$, put
$\lambda_b=f_b(c_0)/f_b(-c_0)$.  Since $p(\delta_b)=-2b$ and
$[N]\delta_b$ has toric coordinate $\lambda_b$ up to inversion,
\begin{equation}\label{eq:absolute-gluing-height}
 \widehat h(\delta_b)=\widehat h_M(\delta_b)
      =\frac{2}{N}h_{\mathrm{Weil}}(\lambda_b).
\end{equation}
The absolute height gap thus gives
$h_{\mathrm{Weil}}(\lambda_b)\ge\eta_0N/2$ for some $\eta_0>0$.
The exceptional set has finite inverse image in $B$, because
$p\circ\delta=-[2]$ has finite fibres.  Taking the union of the two
finite exceptional sets and decreasing the positive constant if
necessary reconciles this bound with
\eqref{eq:explicit-gluing-gap}.  Notice that only linear scaling and
the torus restriction in \eqref{eq:absolute-boundary-scaling} are
needed; no additional comparison of metrics is hidden in this argument.

The same exclusion of torsion cosets gives the torsion finiteness by
absolute Manin--Mumford \cite{Hindry1988,McQuillan1995}.
Thus this constant-generalized-Jacobian example is not a new absolute
Bogomolov case.  It verifies the relative hypotheses in explicit
polynomials and displays the gluing obstruction while the ordinary
Pell orders are unbounded.  The moving-extension results of
Section~\ref{sec:transfer} have a different scope.

\providecommand{\bysame}{\leavevmode\hbox to3em{\hrulefill}\thinspace}
\providecommand{\MR}{\relax\ifhmode\unskip\space\fi MR }
\providecommand{\MRhref}[2]{%
  \href{http://www.ams.org/mathscinet-getitem?mr=#1}{#2}
}
\providecommand{\href}[2]{#2}

\end{document}